\documentclass[12pt]{amsart}

\usepackage{amssymb}
\usepackage{amsmath}
\usepackage{amscd}
\usepackage{amsthm}
\usepackage{latexsym}

\newtheorem{thm}{Theorem}[section]
\newtheorem{prop}[thm]{Proposition}
\newtheorem{lem}[thm]{Lemma}
\newtheorem{cor}[thm]{Corollary}
\newtheorem{rem}[thm]{Remark}
\newtheorem{defin}[thm]{Definition}
\newtheorem{example}[thm]{Example}

\newcommand{\ot}{\otimes}
\newcommand{\op}{\oplus}
\newcommand{\bigop}{\bigoplus}

\newcommand{\lrarr}{\leftrightarrow}
\newcommand{\lolarr}{\longleftarrow}
\newcommand{\lorarr}{\longrightarrow}

\newcommand{\isoarr}{\stackrel{\sim}{\longrightarrow}}

\newcommand{\ra}{\rangle}
\newcommand{\la}{\langle}

\newcommand{\bF}{\mathbb{F}}

\newcommand{\bZ}{\mathbb{Z}}
 
\newcommand{\bCA}{\mathbb{CA}}

\newcommand{\bDA}{\mathbb{DA}}

\newcommand{\mF}{\mathbf{F}}

\newcommand{\Fp}{\mathbb{F}_p} 
\newcommand{\diag}{\mathrm{diag}}

\DeclareMathOperator{\Syl}{Syl}
\DeclareMathOperator{\Aut}{Aut}
\DeclareMathOperator{\Out}{Out}
\DeclareMathOperator{\res}{res}

\DeclareMathOperator{\rad}{rad}

\DeclareMathOperator{\GL}{GL}
\DeclareMathOperator{\SL}{SL}

\newcommand{\mcl}{\mathcal}

\newcommand{\al}{\alpha}
\newcommand{\gam}{\gamma}
\newcommand{\Gam}{\Gamma}

\title[Cohomology of Extraspecial $p$-Group II]
{Representations of the Double Burnside Algebra 
and Cohomology of the Extraspecial $p$-Group II}
\author{Akihiko Hida and Nobuaki Yagita}

\address{Akihiko Hida, Faculty of Education, 
Saitama University,
Shimo-okubo 255, Sakura-ku, Saitama-city, Saitama, Japan}
\email{ahida@mail.saitama-u.ac.jp}

\address{Nobuaki Yagita, 
Department of Mathematics, Faculty of Education, 
Ibaraki University,
Mito, Ibaraki, Japan}
\email{yagita@mx.ibaraki.ac.jp}

\begin{document}
\maketitle

\begin{abstract}
Let $E$ be the extraspecial $p$-group of order $p^3$ and 
exponent $p$ where $p$ is an odd prime. We 
determine the mod $p$ cohomology $H^*(X,\bF_p)$ 
of a summand $X$ in the 
stable splitting of $p$-completed classifying space 
$BE$. In the previous paper [Representations of the double Burnside algebra and cohomology of the extraspecial $p$-group,
J. Algebra 409 (2014) 265-319], we determined these 
cohomology modulo nilpotence. In this paper, 
we consider the whole part of the cohomology. 
Moreover, we consider the stable splittings of $BG$ for 
some finite groups with Sylow $p$-subgroup $E$ 
related with the three dimensional linear group $L_3(p)$.  
\end{abstract}

\section{Introduction}
Let $p$ be an odd prime and $E=p_+^{1+2}$ the extraspecial 
$p$-group of order $p^3$ and exponent $p$. In the previous 
paper \cite{HY}, we determined the composition factor of 
$H^*(E)=(\bF_p\ot H^*(E,\bZ))/\sqrt{(0)}$ as a right 
$A_p(E,E)$-module, where $A_p(E,E)$ is a double Burnside 
algebra of $E$ over $\bF_p=\bZ/p\bZ$. 
In this paper, we consider the whole part of the cohomology 
$H^*(E,\bF_p)$ and determine the composition factor of 
$H^*(E,\bF_p)$ as an $A_p(E,E)$-module. 

The mod $p$ cohomology ring $H^*(E,\bF_p)$ of $E$ is completely known by \cite{L92}, but the structure is very complicated. We shall study $H^*(E, \bF_p)$ through the integral cohomology 
ring $H^*(E,\bZ)$ as in \cite{TY96} and \cite{Y98}. 
Let $H^{even}(E,\bZ)$ (resp. $H^{odd}(E, \bZ)$) 
be the even (resp. odd) degree part of $H^*(E,\bZ)$.  
Let $N=\sqrt{(0)}$ in $\bF_p \ot H^{even}(E,\bZ)$. Then 
we have that 
$\bF_p \ot H^{even}(E,\bZ)/N\cong H^*(E)$. On 
the other hand, the Milnor operator $Q_1$ induces an  
isomorphism
$H^{odd}(E,\bZ) \cong (y_1v,y_2v)H^*(E)$ 
where $(y_1v,y_2v)H^*(E)$ is the ideal of $H^*(E)$ 
generated by $y_1v, y_2v \in H^{2p+2}(E)$
(see the first part of section 2).  

Let  $M=\op M^n$ and $L=\op L^n$ be graded 
$A_p(E,E)$-modules such that $M^n$ and $L^n$ are finite 
dimensional for every $n$. We write as 
$M \lrarr L$ if $M^n$ and $L^n$ have same composition 
factors (with same multiplicity), that is, $[M^n]=[L^n]$ in 
the Grothendieck group $K_0(A_p(E,E))$. 
 
Using this notation, the structure of $H^*(E,\bF_p)$ can be stated 
as follows.  

\begin{thm}\label{t11}
{\rm (1)} 
As $A_p(E,E)$-modules, 
$$H^{even}(E,\bF_p) \lrarr 
H^*(E) \op N \op (y_1v,y_2v)H^*(E)[-2p].$$
{\rm (2)}
As $A_p(E,E)$-modules, 
\begin{eqnarray*}
H^{odd}(E,\bF_p) \lrarr  
(y_1v,y_2v)H^*(E)[-2p+1] \op (N \op H^+(E))[-1]. 
\end{eqnarray*}
\end{thm}

Here, for a graded $\bF_p$-subspace $M$ of $H^*(E)$, 
we denote by $M[i]$ the graded vector space with 
$M[i]^n=M^{n-i}$. 
Since the composition factors of $N$ and $(y_1v,y_2v)H^*(E)$ are determined in 
Proposition \ref{p3N} and Theorem \ref{t3IE}, we 
can get the composition factors of $H^*(E,\bF_p)$ completely.

The indecomposable summands in the complete stable splitting of the $p$-completed classifying space ${BE}_p^{\wedge}$ 
correspond to primitive idempotents in $A_p(E,E)$. Moreover 
they corresponds to simple $A_p(E,E)$-modules. 
We simply write $BE$ for $BE^{\wedge}_p$.  
Let $X$ be a summand in $BE$ which corresponds to a simple $A_p(E,E)$-module $S$. Then the multiplicity of $X$ in 
$BE$ is equal to the dimension of $S$ as an $\Fp$-vector 
space since $\Fp$ is a splitting field for $A_p(E,E)$. 
By results above, we can get the cohomology 
$H^*(X,\Fp)$ (See Remark \ref{r3X}). 

Let $G$ be a finite group with Sylow $p$-subgroup $E$. 
Then the multiplicity of $X$ in $BG$ is equal to 
the dimension of $S[G]$ where $[G]$ is an element of 
$A_p(E,E)$ corresponds to the $(E,E)$-biset $G$. 
See \cite{B}, \cite{BF}, \cite{MP} for details. 

In \cite{Y07}, the second author studied the splitting 
of $BG$ for various finite groups $G$ whose Sylow $p$-subgroup 
is $E$ and $p$-local finite groups on $E$. 
In this paper, we consider 
the stable splitting for groups related with the linear group 
$L_3(p)$ which were not treated in \cite{Y07} in general. 
We use some simple $A_p(E,E)$-submodules of 
$H^*(E)$ and determine the multiplicity of summands  
in $BG$ for $G=L_3(p)$, $L_3(p):2$, $L_3(p).3$, 
$L_3(p).S_3$ (Theorem \ref{t4WT}, \ref{t4WTw}, 
\ref{t4WH}, \ref{t4WHw}). 

Combining these results and results in \cite{Y07}, 
we have the complete information 
on the stable splitting of finite groups or $p$-local finite groups 
which have at least two $\mcl{F}$-radical maximal 
elementary abelian $p$-subgroups in $E$, by the classification 
in \cite{RV}, where $\mcl{F}$ is a fusion system of $G$. 

In particular, for $p=7$, we obtain a diagram which 
describes inclusions of some fusion systems 
and stable splitting  (Theorem \ref{t4p7-2}). This result supplements the results of \cite[section 9]{Y07}, in which the splitting of sporadic simple groups are mainly studied. 

In section 2, 
we review the main results of \cite{HY} which will be used in 
section 3. In section 3, we prove Theorem \ref{t11} and 
determine the structures of ideals $N$ and 
$(y_1v,y_2v)H^*(E)$. 
In section 4, we consider $H^*(G)$ and the stable splitting 
for finite group $G$ which 
has a Sylow $p$-subgroup $E$. Finally, in section 5, we 
consider the case $p=3$ and state some remarks.


\section{Preliminary results on $H^*(E)$}
In this section, we quote some results from \cite{HY}.
Let $p$ be an odd prime. Let 
$$E=\la a,b,c~|~[a,b]=c,\ a^p=b^p=c^p=[a,c]=[b,c]=1\ra$$ 
be the extraspecial $p$-group of order $p^3$ and 
exponent $p$. Let 
$A_i=\la c, ab^i\ra$ for $0 \leq i \leq p-1$ and 
$A_{\infty}=\la c, b\ra$. Then 
$$\mcl{A}(E)=\{A_0, A_1, \dots, A_{p-1}, A_{\infty}\}$$
is the set of all maximal elementary abelian 
$p$-subgroups of $E$. 

The cohomology of $E$ is known by \cite{L91}, \cite{L92}, 
\cite{Y07}. In particular, 
$H^*(E)=(\bF_p\ot H^*(E,\bZ))/\sqrt{(0)}$ is generated by  
$$y_1,~y_2,~C,~v$$
with 
$$\deg y_i=2,~\deg C=2p-2,~\deg v=2p$$  
subject to the following relations:
$$y_1^py_2-y_1y_2^p=0,~ Cy_i=y_i^{p},~
C^2=y_1^{2p-2}+y_2^{2p-2}-y_1^{p-1}y_2^{p-1}.$$
We set $V=v^{p-1}$ and $Y_i=y_i^{p-1}$.  

Let $R$ be a subalgebra of $H^*(E)$ and 
$x_1, \dots, x_r$ elements of $H^*(E)$. 
We set 
$$R\{x_1, \dots, x_r\}=\sum_{i=1}^r Rx_i$$
if $x_1, \dots, x_r$ are linearly independent over $R$. 
Moreover, if $W=\sum_{i=1}^r\bF_p x_i$ is a $\bF_p$-vector 
space spanned by $x_1,\dots, x_r$, then 
we set 
$$R\{W\}=R\{x_1, \dots, x_r\}.$$

We consider the action of $\Out(E)=\GL_2(\bF_p)$ on 
$H^*(E)$. 
Let $S^i$ be the homogeneous part of degree $2i$ 
in $\Fp[y_1,y_2]$.   
Then $p(p-1)$ simple $\Fp\Out(E)$-modules 
$$S^iv^q \cong S^i\otimes (\det)^q\quad  
(0\le i\le p-1,0\le q\le p-2)$$
give the complete set of representatives of nonisomorphic  
simple $\Fp\Out(E)$-modules. 
Let us write 
\[ \bCA=\Fp[C,V] \]
and 
\[\bDA=\Fp[D_1,D_2]\]
where $D_1=C^p+V$, $D_2=CV$. Then $\bCA=H^*(E)^{\Out(E)}$, 
the $\Out(E)$-invariants, and the restriction map induces an 
isomorphism 
$$\bDA \isoarr H^*(A)^{\Out(A)}$$
for all $A \in \mcl{A}(E)$.  

Let 
\[T^i=\Fp\{y_1^{p-1}y_2^i,\ y_1^{p-2}y_2^{i+1}, 
\dots , y_1^iy_2^{p-1}\} \] 
for $1 \leq i \leq p-2$. 
Then  $S^{p-1+i}=CS^i+ T^i$. 
The $\Fp$-subspace $CS^i$ is a 
$\GL_2(\Fp)$-submodule of $CS^i+T^i$ and  
\[(CS^i+T^i)/CS^i \cong (S^{p-1-i}\otimes {\det}^i).\]
Moreover we have the following expression \cite[Theorem 4.4]{HY}: 
\[ H^*(E)=
\Fp[C,v]\{(\bigoplus_{i=0}^{p-1}S^i)
\oplus(\bigoplus_{i=1}^{p-2}T^i)\}
=
\bCA \{
\bigoplus_{i=0}^{p-2}\bigoplus_{q=0}^{p-2}
(S^iv^q \oplus T^iv^q))\} \]
where $S^0=\Fp$ and $T^0=S^{p-1}$. 

Let $C_p$ be a cyclic group of order $p$ and let 
$U_i=H^{2i}(C_p, \Fp)$ ($0 \leq i\leq p-2$). Then 
$U_i$ are simple $\Fp\Out(C_p)$-modules.  
Let $A \in \mcl{A}(E)$ be a maximal elementary abelian $p$-subgroup of 
$E$. Let $S(A)^{i}=H^{2i}(A)$. Then 
$S(A)^{i} \ot {\det}^q$ ($0 \leq i \leq p-1$, $0 \leq q \leq p-2$) 
are  simple modules for 
$\Out(A)=\GL_2(\bF_p)$. 

Let $P$ be a general finite $p$-group and $A_p(P,P)$ 
the double Burnside algebra of $P$ over $\Fp$. 
The simple $A_p(P,P)$-modules corresponds to some pairs 
$(Q,V)$ where $Q \leq P$ and $V$ is a simple 
$\bF_p\Out(Q)$-module, see \cite{BF}, \cite{Bo}, \cite{MP}. 
In this paper, we denote the simple $A_p(P,P)$-module corresponds to the pair $(Q,V)$ by $S(P,Q,V)$. 

On the other hand, Dietz and Pridy \cite{DP} studied the stable splitting of $BE$ and determined the multiplicity of each summand. 
In particular, their result implies the classification of simple $A_p(E,E)$-modules. 

\begin{prop}[\cite{DP}, {\cite[Proposition 10.1]{HY}}]\label{p2simple}
The simple $A_p(E,E)$-modules are given as follows: \\
{\rm(1)} $S(E,E,S^i\ot {\det}^q)$ for 
$0 \leq i \leq p-1$, $0 \leq q \leq p-2$, 
$$\dim S(E,E,S^i\ot {\det}^q)=i+1.$$
{\rm(2)} $S(E,A,S(A)^{p-1}\ot {\det}^q)$ for $0\leq q \leq p-2$,
$$\dim S(E,A,S(A)^{p-1}\ot {\det}^q)=p+1.$$
{\rm(3)} $S(E,C_p,U_i)$ for $0 \leq i \leq p-2$, 
$$\dim S(E,C_p,U_i)=
\begin{cases}
 p+1 & (i=0) \\
i+1  &  (1 \leq i \leq p-2).
\end{cases}$$
{\rm(4)} $S(E,1, \bF_p)$, $\dim S(E,1, \bF_p)=1$.   
\end{prop}

To describe the composition factor of $H^*(E)$ as an 
$A_p(E,E)$-module, we need the following $\bF_p$-subspace 
of $H^*(E)$. 

\begin{defin}\label{d2gam} 
Let $S$ be a simple $A_p(E,E)$-module. Let 
$\Gam_S$ be the following $\bF_p$-subspace of 
$H^*(E)$:\\
{\rm(1)} If $S =S(E,C_p, U_i)$, then
$$\Gam_S=\left\{
\begin{array}{lc}
\Fp[C] \{\Fp C+S^{p-1}\}& (i=0) \\
\Fp[C]\{S^i\} & (1 \leq i \leq p-2).
\end{array}\right.$$
{\rm (2)} If $S = S(E,A,S(A)^{p-1}\ot{\det}^q)$, then  
$$\Gam_S= 
\left\{
\begin{array}{cc}
\bDA \{\op_{0 \leq j \leq p-1}D_2C^j(\Fp C+S^{p-1}))\} & (q=0) \\
\bDA \{\op_{0 \leq j \leq p-1}v^qC^j(CS^q+T^q)\} & (1 \leq q \leq p-2).
\end{array}\right.$$
{\rm(3)}  
\begin{eqnarray*}
\Gam_S=
\begin{cases} 
\bDA^+ & (S=S(E,E,S^0)) \\
\bCA\{v^q\}& (S=S(E,E,\det^q),~ 1 \leq q \leq p-2) \\
\bDA \{VS^{p-1}\} & (S=S(E,E,S^{p-1}) \\
\bCA\{v ^q S^{p-1}\} & (S=S(E,E,S^{p-1}\ot {\det}^q), 
1 \leq q \leq p-2)
\end{cases}
\end{eqnarray*}
{\rm(4)} Let 
$$S=S^iv^q, \quad T=T^{p-i-1}v^{s}$$
for $1\leq i \leq p-2$, $0 \leq q \leq p-2$, where 
$s \equiv i+q \pmod{p-1}$, $0\leq s \leq p-2$. 
Let $\Gam_{S(E,E,S^i \ot \det^q)}$ be the following $\bF_p$-subspace: $$\begin{array}{lcll}
\bCA \{VS\} &\op& \bDA \{ VT\} & (q \equiv 2i \equiv 0) \\
\bCA \{ VS\} &\op& \bCA\{T\}
  & (q \equiv 0,~2i \not\equiv 0) \\
\bDA\{ S \} &\op& \bDA\{ VT\} & (i=q,~ 3i \equiv 0) \\
\bDA \{S\}   &\op& \bCA\{T\}  & (i=q,~ 3i \not\equiv 0)\\
\bCA \{ S\}  &\op& \bDA\{VT\} & (q \ne 0,~i \ne q,~q+2i \equiv 0) \\
\bCA\{S\}  &\op& \bCA \{T\}  & (q \ne 0,~i \ne q,~q+2i \not\equiv 0) . 
\end{array}$$
{\rm(5)} 
$$\Gam_{S(E,1,\bF_p)}=\bF_p=H^0(E).$$
\end{defin}

The following theorem is the main result of \cite{HY}. 
If $S$ is a simple $A_p(E,E)$-module, then there exists an 
idempotent $e_S$ such that $Se_S=S$ and $S'e_S=0$ for 
any simple module $S' \not\cong S$. We call $e_S$ an 
idempotent which corresponds to $S$. 

\begin{thm}[{\cite[Theorem 10.2, 10.3, 10.4, 10.5]{HY}}]\label{t2HY} 
Let $S$ be a simple $A_p(E,E)$-module. Then there exists 
an idempotent $e_S$ which corresponds to 
$S$ such that 
$$H^*(E)e_S=\Gam_Se_S \cong \Gam_S.$$
\end{thm}

If we see the minimal degree of non zero part 
of $\Gam_S$, we have the following corollary. 

\begin{cor}\label{c2deg} 
Every simple $A_p(E,E)$-module appears as a composition 
factor in $H^{2n}(E)$ for some $n \leq (p+2)(p-1)$. 
\end{cor}

Let $\Gam^n_S=\Gam_S\cap H^n(E)$ be the degree $n$ part of 
$\Gam_S$. Then by Theorem \ref{t2HY}, 
$$\sum_S\dim \Gam_S^n = \dim H^n(E)$$
for any $n \geq 0$. In fact we have the 
following. 

\begin{prop}\label{p2sum} 
$H^*(E)$ is a direct sum of $\Fp$-subspaces 
$\Gam_S$ where $S$ runs over the representatives 
of the isomorphism classes of simple $A_p(E,E)$-modules. 
\end{prop}

\begin{proof}
By Theorem \ref{t2HY}, it suffices to show that $H^*(E)= \sum\Gam_S$. We shall show that 
$\bCA \{S^iv^q\}$ $(0 \leq i \leq p-1, ~ 0\leq q \leq p-2)$ and 
$\bCA \{T^iv^q\}$ $(1 \leq i \leq p-2, ~0 \leq q \leq p-2)$ are 
contained in $\sum \Gam_S$.

First consider $\bCA \{S^iv^q\}$. 
If $(i,q)=(0,0)$,  
\begin{eqnarray*}
\bCA=\bF[C,D_1]&=&\bF[D_1]\op \bCA\{C\}\\
&=& \bF[D_1]\op \bF[C]\{C\}\op \bCA\{D_2\}\\
&=& \bF[D_1]\op \bF[C]\{C\}\op 
\sum_{j=0}^{p-1}\bDA\{D_2C^{j+1}\}\op \bDA\{D_2\}\\
&=& \bDA \op \bF[C]\{C\} \op \sum_{j=0}^{p-1}\bDA\{D_2C^{j+1}\}
\end{eqnarray*}
and this is contained in the sum of the subspaces 
of Definition \ref{d2gam} (1)(2)(3)(5). 
If $(i,q)=(0,q)$, $1 \leq q \leq p-2$, or 
$(i,q)=(p-1,q)$, $1 \leq q \leq p-2$, then 
$\bCA v^q$ and $\bCA S^{p-1}v^q$ are contained in the subspace of Definition \ref{d2gam} (3). 
If $(i,q)=(p-1,0)$, then  
\begin{eqnarray*}
\bCA&=&\bF[C]\op\bCA\{V\} \\
&=&\bF[C]\op \bDA\{V\}\op \bDA\{D_2, D_2C, \dots, D_2C^{p-1}\}
\end{eqnarray*}
and we have 
$$\bCA\{ S^{p-1}\}=
\Fp[C]\{S^{p-1}\}\op \bDA\{VS^{p-1}\}\op
(\op_{j=0}^{p-1}\bDA\{D_2C^jS^{p-1}\}).$$
This is contained in the sum of subspaces of Definition 
\ref{d2gam} (1)(2)(3).  

Consider $(i,q)$ $(1 \leq i \leq p-2, 0 \leq q \leq p-2)$. 
If $q=0$, then
$$\bCA\{S^i\}=\Fp[C]\{S^i\} \op \bCA\{VS^i\}$$
and this is contained in the sum of subspaces of 
Definition \ref{d2gam} (1)(4). 
If $i=q$, then 
$$\bCA \{S^iv^i\}=
\bDA\{S^iv^i\} \op 
(\op_{j=0}^{p-1} \bDA\{C^{j+1}S^iv^i\})$$
and this is contained in the sum of the subspaces of Definition 
\ref{d2gam} (2)(4). 
If $q \ne0$ and $i \ne q$, then 
$\bCA\{S^iv^q\}$ is contained in the subspace of 
Definition \ref{d2gam} (4). 

Next, consider $T^kv^m$ ($1 \leq k\leq p-2$, $0 \leq m \leq p-2$). 
Let $i=p-k-1$, $q \equiv m+k \bmod (p-1)$, 
$0 \leq q \leq p-2$, $s=m$. Then 
$$T^kv^m=T^{p-i-1}v^s$$
where $1 \leq i \leq p-2$, $0 \leq q \leq p-2$, $s\equiv i+q \bmod(p-1)$ 
and 
$$q+2i \equiv m+k+2(p-k-1) \equiv m-k \bmod (p-1).$$
If $k \ne m$, then $q+2i \not\equiv 0$ and 
$\bCA\{T^kv^m\}$ is contained in the subspace of 
Definition \ref{d2gam} (4). 
If $k=m$, then $q+2i \equiv 0$ and 
$$\bCA\{T^kv^m\}=\bDA\{\op_{j=0}^{p-1}T^kC^jv^m\} \op 
\bDA\{T^kVv^m\}$$
since $\bCA=\bDA\{1,C, \dots, C^{p-1},V\}$. 
This is contained in the sum of the subspaces 
of Definition \ref{d2gam} (2)(4). 
\end{proof}

\section{Composition factors of  $H^*(E,\bF_p)$}
In this section, we study the $A_p(E,E)$-module structure 
of $H^*(E,\bF_p)$. First, we shall prove Theorem \ref{t11}. 
The even degree part $H^{even}(E,\bZ)$ of integral 
cohomology ring is generated by 
$$y_1,y_2, b_2, \dots, b_{p-2},C,v$$
with 
$$\deg y_i=2,~\deg b_i=2i$$
subject to the following relations:
$$py_i=pb_j=pC=0,~p^2v=0,$$ 
$$y_1y_2^p-y_1^py_2=0,$$
$$y_ib_k=b_kb_j=Cb_j=0,$$
$$y_iC=y_i^p,~
C^2=y_1^{2p-2}+y_2^{2p-2}-y_1^{p-1}y_2^{p-1}$$
by \cite{Lewis} or \cite[Theorem 3]{L91} (see \cite{TY96} also). 
In particular, $p^2(H^{2n}(E,\bZ))=0$ for any $n>0$.  

On the other hand, the odd degree part of integral 
cohomology ring $H^{odd}(E,\bZ)$ is annihilate by $p$ and so it 
is considered as an $\bF_p[y_1,y_2,v]$-module. As an  
$\bF_p[y_1,y_2,v]$-module, $H^{odd}(E,\bZ)$ is generated 
by two elements $a_1$ and $a_2$
with $\deg a_i=3$ 
subject to the following relations:
$$y_1a_2-y_2a_1=0,~y_1^pa_2-y_2^pa_1=0.$$

Let $H^*(E,\bZ) \lorarr H^*(E, \bF_p)$ be the natural 
map induced by $\bZ \lorarr \bF_p$.   
We use the same letters for the images of $y_i, b_j,C,v$ 
in $H^*(E,\bF_p)$. Then 
$$N=\bF_p[v]\{b_2,\dots, b_{p-2}\}=\sqrt{0}$$ in 
$$H^{even}(E,\bZ)/pH^{even}(E,\bZ)=\bF_p\ot_{\bZ}H^{even}(E,\bZ).$$ 

Since
$$H^*(E)=(\bF_p\ot_{\bZ}H^*(E,\bZ))/\sqrt{0}=
(\bF_p\ot_{\bZ}H^{even}(E,\bZ))/N,$$
there is a short exact sequence of $A_p(E,E)$-modules,
\[0 \lorarr N \lorarr H^{even}(E,\bZ)/pH^{even}(E,\bZ) \lorarr H^*(E) \lorarr 0.
\tag{3.1} \] 
On the other hand, since $pH^{odd}(E,\bZ)=0$, there 
is a short exact sequence of $A_p(E,E)$-modules,
\[ 0 \lorarr H^{even}(E,\bZ)/pH^{even}(E,\bZ) 
\lorarr H^{even}(E,\bF_p) 
\lorarr H^{odd}(E,\bZ)[-1] \lorarr 0.
\tag{3.2} \]

Let $(y_1v,y_2v)H^{even}(E,\bZ)$ (resp. $(y_1v,y_2v)H^*(E)$) 
be the ideal of $H^{even}(E,\bZ)$ (resp. $H^*(E)$) 
generated by $y_1v$ and $y_2v$. 
Since $py_i=0$ and $y_iN=0$, it follows that 
$$(y_1v,y_2v)H^{even}(E,\bZ) \cong  
(y_1v,y_2v)H^*(E).$$

Here we use Milnor's primitive operator  
$Q_1=P^1\beta-\beta P^1$ on $H^*(-,\bF_p)$. This operator   
induces a map $Q_1$ on $H^*(-,\bZ)$ such that the following 
diagram commutes: 
$$
\begin{CD}
H^{odd}(E,\bZ) @>{Q_1}>> H^{even}(E,\bZ) \\
@VVV                       @VVV                 \\
H^{odd}(E,\bF_p) @>{Q_1}>> H^{even}(E, \bF_p).  
\end{CD}$$
Moreover, $Q_1$ induces an isomorphism of  
$A_p(E,E)$-modules, 
$$Q_1: H^{odd}(E,\bZ) \isoarr (y_1v,y_2v)H^{even}(E,\bZ)\cong 
(y_1v,y_2v)H^*(E)$$ 
$$Q_1(a_i)=y_iv,$$
(see \cite[section 1]{Y98}). 
Then  we have
\[H^{odd}(E,\bZ)\cong (y_1v,y_2v)H^*(E)[-2p+1]
\tag{3.3}\] 
and the proof of the first part of Theorem \ref{t11} is 
completed by the exact sequences (3.1) and (3.2).

Next we consider the odd degree part $H^{odd}(E,\bF_p)$. 
Let $K=\{x \in H^{even}(E,\bZ)~|~px=0\}$. Then 
there exists a short exact sequence of 
$A_p(E,E)$-modules, 
\[0 \lorarr H^{odd}(E,\bZ) \lorarr H^{odd}(E,\bF_p) \lorarr 
K[-1] \lorarr 0. \tag{3.4} \]
Let $H=H^{even}(E,\bZ)\cap H^+(E,\bZ)$. 
Since $p^2H=0$, 
$pH \subset K$. Moreover, $K$, $pH$ and $H/K$ are 
$A_p(E,E)$-modules. 

Since the map $p :H \lorarr H$ is a homomorphism of $A_{\bZ}(E,E)$-modules, 
$$H/K \cong pH$$
as $A_{\bZ}(E,E)$-modules. Since these are modules for 
$A_p(E,E)=A_{\bZ}(E,E)/pA_{\bZ}(E,E)$, these are isomorphic 
as $A_p(E,E)$-modules.  Hence we have 
$$K \lrarr pH \op K/pH \lrarr H/K \op K/pH \lrarr H/pH.$$
Moreover, since
$$H/pH=\bF_p \ot H \lrarr N \op H^+(E),$$
the proof of the second part of Theorem \ref{t11} is 
completed by (3.3) and (3.4). 
\vspace{.5cm}

Next we shall see the structure of $N$. 
Note that $\res^E_A(b'_i)=0$ for any $i$ and any maximal 
elementary abelian $p$-subgroup 
$A$ of $E$. Since the action of $g \in \GL_2(\bF_p)$ is 
given by 
$$g^*(b_i)=\det(g)^ib_i,~g^*(v)=\det(g)v$$
(see \cite[Theorem 3]{L91}), 
$N$ is a direct sum of simple $A_p(E,E)$-modules isomorphic 
to $S(E,E, \det^i)$ for $0 \leq i \leq p-2$. 
Hence we have the following: 

\begin{prop}\label{p3N} 
Let 
$$N_q=\bF_p\{v^kb_i'~|~k \geq 0,~k+i \equiv q \bmod (p-1)\}$$
for $0 \leq q \leq p-2$. Then 
$$N=\bigop_{0\leq q\leq p-2} N_q$$
and 
$$N_q \cong \bigop S(E,E,{\det}^q)$$
as $A_p(E,E)$-modules. 
\end{prop}   

Next, we shall consider the structure of the ideal $(y_1v,y_2v)H^*(E)$. 
Let $I=(y_1v,y_2v)H^*(E)$.  
Let $\Gam_S$ be the $\bF_p$-subspace defined in Definition 
\ref{d2gam} for each simple $A_p(E,E)$-module $S$. 
We shall show that 
$$Ie_S \cong I\cap \Gam_S$$
for an idempotent $e_S$ which corresponds to $S$ and 
determine the $\bF_p$-subspace $I\cap \Gam_S$  
explicitly.

\begin{lem}\label{l3L} 
Let 
$$L=\bF_p[y_1,y_2,C] \op \bF_p[V]\{v, \dots, v^{p-2},
Cv, \dots,Cv^{p-2}\} \op \bF_p[D_1]\{D_1\} \op 
\bF_p[D_1]\{D_2\}$$
where $\bF_p[y_1,y_2,C]$ is the subalgebra  of 
$H^*(E)$ generated by $y_1, y_2$ and $C$. 
Then 
$$H^*(E)=I\op L.$$
\end{lem}

\begin{proof}
Let $(y_1,y_2,C)$ be the ideal of $H^*(E)$ generated by 
$y_1, y_2$ and $C$. Since 
$$D_1=C^p+V \equiv V \bmod (y_1,y_2,C),$$
we have 
\begin{eqnarray*}
H^*(E) &=&
\bF_p[v] \op (y_1,y_2,C) \\
&=& \bF_p \op \bF_p[V]\{v,\dots, v^{p-2}\}\op \bF[V]\{V\} 
\op (y_1,y_2,C) \\
&=&
\bF_p \op \bF_p[V]\{v,\dots, v^{p-2}\}\op \bF[D_1]\{D_1\} 
\op (y_1,y_2,C).
\end{eqnarray*}
On the other hand, since
$$D_2D_1=CV(C^p+V) \equiv D_2V \bmod I ,$$
it follows that 
\begin{eqnarray*}&&
\bF_p \op (y_1, y_2,C) \\ &=&
\bF_p[y_1,y_2,C] \op \bF_p[v]\{Cv\} \op I \\
&=&
\bF_p[y_1,y_2,C] \op  \bF_p[V]\{Cv,\dots,Cv^{p-2}\} 
\op \bF_p[V]\{CV\}\op I \\
&=&
\bF_p[y_1,y_2,C] \op  \bF_p[V]\{Cv,\dots,Cv^{p-2}\} 
\op \bF_p[D_1]\{D_2\}\op I
\end{eqnarray*}
and we have $H^*(E)= I \op L$. 
\end{proof}

\begin{lem}\label{l3IEGam} 
For each simple $A_p(E,E)$-module $S$, we have
$$\Gam_S=(I \cap \Gam_S) \op (L \cap \Gam_S)$$
where $I=(y_1v,y_2v)H^*(E)$ and $L$ is an $\Fp$-subspace 
defined in Lemma \ref{l3L}. 
Moreover, \\
{\rm(1)} If $S =S(E,C_p, U_i)$, then 
$I \cap \Gam_S=0$. \\
{\rm (2)} If $S = S(E,A,S^{p-1}\ot \det^q)$, then  
$$I \cap\Gam_S=\Gam_S= 
\left\{
\begin{array}{cc}
\bDA \{\bigop_{0 \leq j \leq p-1}D_2C^j(\Fp C+S^{p-1})\} & (q=0) \\
\bDA \{\bigop_{0 \leq j \leq p-1}v^qC^j(CS^q+T^q)\} & (1 \leq q \leq p-2).
\end{array}\right.$$
{\rm(3)} 
$$
I \cap \Gam_S=
\begin{cases}
\bDA \{D_2^2\} & (S=S(E,E,S^0)) \\ 
\bCA\{C^2v^q\} & (S=S(E,E,\det^q),~1 \leq q \leq p-2) \\
\bDA \{VS^{p-1}\} & (S=S(E,E,S^{p-1})) \\
\bCA\{v ^q S^{p-1}\} & (S=
S(E,E,S^{p-1}\ot {\det}^q),~1 \leq q \leq p-2)
\end{cases}
$$
{\rm(4)} Let 
$$S=S^iv^q, \quad T=T^{p-i-1}v^{s}$$
for $1\leq i \leq p-2$, $0 \leq q \leq p-2$, where 
$s \equiv i+q \pmod{p-1}$, $0\leq s \leq p-2$. 
Then $I \cap \Gam_{S(E,E,S^i \ot \det^q)}$ is the following $\bF_p$-subspace: 
$$\begin{array}{lcll}
\bCA \{VS\} &\op& \bDA \{ VT\} & (q \equiv 2i \equiv 0) \\
\bCA \{ VS\} &\op& \bCA\{T\}
  & (q \equiv 0,~2i \not\equiv 0) \\
\bDA\{ S \} &\op& \bDA\{ VT\} & (i=q,~ 3i \equiv 0) \\
\bDA \{S\}   &\op& \bCA\{T\}  & (i=q,~ 3i \not\equiv 0,~
2i \not\equiv 0)\\
\bDA\{S\}  &\op&  \bCA\{VT\}  & (i=q,~ 3i \not\equiv 0,~
2i \equiv 0)\\
\bCA \{ S\}  &\op& \bDA\{VT\} & (q \ne 0,~i \ne q,~q+2i \equiv 0) \\
\bCA\{S\}  &\op& \bCA \{T\}  & (q \ne 0,~i \ne q,~q+2i 
\not\equiv 0,~ i+q \not\equiv 0) \\
\bCA\{S\}  &\op&  \bCA\{VT\}  & (q \ne 0,~i \ne q,~q+2i \not\equiv 0, ~i+q \equiv 0). 
 \end{array}$$
{\rm(5)} $I \cap \Gam_{S(1,1,\bF_p)}=0$.  
\end{lem}

\begin{proof}
(1) If $S=S(E,C_p,U_i)$ $(0 \leq i \leq p-2)$, then 
$I\cap \Gam_S=0$ and $\Gam_S \subset L$ by 
Definition \ref{d2gam}. \\
(2) If $S= S(E,A,S^{p-1}\ot \det^q)$ $(0 \leq q \leq p-2)$, 
then $\Gam_S \subset I$ since $D_2C=C^2V \in I$. 
\\
(3) If $S=S(E,E,S^0)$, then $\Gam_S=\bDA^+$, 
$$\bDA^+=\bDA\{D_2^2\}\op\Fp[D_1]\{D_1,D_2\}.$$
Since $\bDA\{D_2^2\} \subset I$ and 
$\Fp[D_1]\{D_1,D_2\} \subset L$, 
we have 
$$\bDA^+=(I\cap \Gam_S)\op (L\cap \Gam_S)$$
and $I\cap \Gam_S=\bDA\{D_2^2\}$. 

If $S=S(E,E,\det^q)$ $(1 \leq q \leq p-2)$, 
then $\Gam_S=\bCA\{v^q\}$, 
$$\bCA\{v^q\}=\bCA \{C^2v^q\}\op \Fp[V]\{v^q,Cv^q\}.$$
Since $C^2v^q \in I$ and $\Fp[V]\{v^q,Cv^q\} \subset L$, 
it follows that 
$$\bCA\{v^q\}=(I \cap \Gam_S) \op (L \cap \Gam_S)$$
and $I \cap \Gam_S=\bCA\{C^2v^q\}$. 

If $S=S(E,E,S^{p-1})$, then 
$\Gam_S=\bDA \{VS^{p-1}\}$.
If $S=S(E,E,S^{p-1}\ot {\det}^q)$ $(1 \leq q \leq p-2)$, 
$\Gam_S=\bCA\{v ^q S^{p-1}\}$. 
In these cases, $\Gam_S \subset I$. 
\\
(4) Let $S=S^iv^q$ ($1 \leq i \leq p-2$, $0 \leq q \leq p-2$). 
If $q=0$ then $SV\subset I$. If $q \ne 0$ then 
$S\subset I$. Hence the first term of $\Gam_S$ is 
contained in $I$. 

Let $T=T^{p-i-1}v^s$, $s \equiv i+q \bmod (p-1)$, 
$0 \leq s \leq p-2$. 
If $s \equiv 0$, then $VT \subset I$. 
If $i+q \not\equiv 0$, then $T \subset I$. 
Hence the second term of $\Gam_S$ is contained in 
$I$ unless 
$i=q,~3i \not\equiv 0,~2i \equiv 0$, or 
$q \ne 0,~i \ne q,~ q+2i \not\equiv 0,~i+q \equiv 0$. 
In these cases, 
$$\bCA\{T\}=\Fp[C]\{T\} \op \bCA\{VT\}$$
where $\bCA\{VT\} \subset I$, $\Fp[C]\{T\}\subset L$. 
Hence 
$$\bCA\{T\}=\bCA\{T\}\cap I \op \Fp[C]\{T\}\cap I$$
and $\bCA\{T\}\cap I=\bCA\{VT\}$. 
\end{proof}

\begin{lem}\label{l3IEsum} 
Let $I=(y_1v,y_2v)H^*(E)$. Then  
$$I= \bigop_S (I\cap \Gam_S).$$
\end{lem}

\begin{proof}
First, $H^*(E)= \op_S \Gam_S=I \op L$ by Proposition 
\ref{p2sum} and Lemma \ref{l3L}. On the other hand, 
$$\Gam_S=(I\cap \Gam_S) \op (L \cap \Gam_S)$$
by Lemma \ref{l3IEGam}.
Hence 
\begin{eqnarray*}
H^*(E)&=&\op \Gam_S=
\op ((I \cap\Gam_S)\op(L \cap \Gam_S))\\
&=&(\op (I \cap \Gam_S))\op(\op (L\cap \Gam_S))
\subset I \op L =H^*(E).
\end{eqnarray*}
Hence we have 
$$I= \bigop_S(I\cap \Gam_S).$$
\end{proof}

Now, we determine the $\bF_p$-vector space $Ie_S$ for 
any simple $A_p(E,E)$-module $S$.  

\begin{thm}\label{t3IE} 
Let $S$ be a simple $A_p(E,E)$-module. 
Then there exists an idempotent $e_S$ corresponding to $S$ 
such that 
$$Ie_S=(I\cap \Gam_S)e_S\cong
I \cap \Gam_S.$$
\end{thm}

\begin{proof}
By Theorem \ref{t2HY}, $\Gam_S \cong \Gam_Se_S$ 
for some idempotent corresponding to $S$. 
Hence $e_S$ induces an isomorphism 
$$I\cap \Gam_S \cong (I\cap \Gam_S)e_S.$$
For a graded $\Fp$-subspace $M \subset H^*(E)$, 
let $M^n=H^n(E) \cap M$. 
Then 
\begin{eqnarray*}
\dim I^n&=& \sum_S \dim (I^n)e_S \geq 
\sum_S \dim (I^n \cap\Gam_S)e_S \\
&=& \sum_S \dim  I^n \cap\Gam_S=\dim I^n.
\end{eqnarray*}
The last equality follows from Lemma \ref{l3IEsum}. 
Hence we have 
$$Ie_S=(I\cap\Gam_S)e_S.$$
\end{proof}

\begin{rem}\label{r3X}
{\rm
Let $X$ be the indecomposable summand in the complete 
stable splitting of $BE$ which corresponds to a simple 
$A_p(E,E)$-modules $S$. Let $e_S$ be an idempotent which 
correspond to $S$ as above. Then we can get the 
$\Fp$-vector space 
$$H^*(E, \Fp)e_S\cong H^*(\vee X, \Fp) \cong 
\bigoplus^d H^*(X, \Fp)$$
where $\vee X$ is a wedge sum of $d=\dim S$ copies of $X$, 
from Theorem \ref{t11}, Proposition \ref{p3N} and Theorem \ref{t3IE}. 
}
\end{rem}

\begin{cor}\label{c3deg} 
Every simple $A_p(E,E)$-module appears as a composition 
factor in $H^{2n}(E, \bF_p)$ for some $n \leq p^2-2$. 
\end{cor}

\begin{proof}
Let $H(p^2-2)=\op_{n=0}^{p^2-2}H^{2n}(E,\bF_p).$ 
For a simple $A_p(E,E)$-module $S$, 
let $2\gam(S)$ be the lowest degree 
such that $(\Gam_S \cap I)^{2\gam(S)} \ne 0$.
If $2\gam(S) \leq 2(p+2)(p-1)$, then $2(\gam(S)-p) \leq 2(p^2-2).$ 
Since $S$ appears in the degree $2(\gam(S)-p)$ part of 
$I[-2p]$, it follows that $S$ appears in 
$H^{2(\gam(S)-p)}(E,\bF_p)$ by Theorem \ref{t11}. 
Hence $S$ appears in $H(p^2-2)$. In particular, 
if $\Gam_S \leq I$, then $S$ appears in $H(p^2-2)$ by Corollary 
\ref{c2deg}. This implies that simple modules 
$$S(E,A,S(A)^{p-1} \ot {\det}^q) (0 \leq q \leq p-2), ~ 
S(E,E,S^{p-1}\ot {\det}^q)(0\leq q \leq p-2)$$
appear in $H(p^2-2)$. 

On the other hand, since the degrees of 
$C^2v^q$, $(1\leq q \leq p-2)$ and $VS^i$, $S^iv^q$, ($1 \leq i \leq p-2$, $0 \leq q \leq p-2$) are all smaller than 
$\deg D_2S^{p-1}=2(p+2)(p-1)$, we have that 
$S(E,E,{\det}^q)$, $(1 \leq q \leq p-2)$ and 
$S(E,E, S^i\ot {\det}^q)$, ($1 \leq i\leq p-2$, $0 \leq q \leq p-2$) 
appear in $H(p^2-2)$. 

Moreover, since $S(E,C_p,U_i)$ ($0 \leq i\leq p-2$) appears 
in $H^2(E) \op \cdots \op H^{2(p-1)}(E)$, 
it appears in $H(p^2-2)$. Finally, we consider $S(E,E,S^0)$. 
Since it appears in $H^{2p(p-1)}(E)$ and  
$2p(p-1) (=\deg D_1) \leq 2(p^2-2)$, $S(E,E,S^0)$ 
appears in $H(p^2-2)$. This completes the proof. 
\end{proof}



\section{Stable splitting of groups related to $L_3(p)$}

In this section,  we consider the stable splitting of $BG$ for 
groups $G$ having $E$ as a Sylow $p$-subgroup, in particular 
the linear group $L_3(p)$ and its extensions. 

Benson and Fechbach \cite{BF},  Martino and Priddy \cite{MP} 
prove the following theorem on complete stable splitting.
Let $P$ be a finite $p$-group. If $G$ is a finite group which contains $E$, then $G$ is considered as an $(E,E)$-biset. 
We denote by $[G]$ the element of $A_p(E,E)$ corresponding 
to $G$. 
Let $S(P,Q,V)$ be the simple $A_p(P,P)$-module which 
corresponds to $(Q,V)$ where $Q$ is a subgroup of $P$ and 
$V$ is a simple $\Fp \Out(P)$-module. 

\begin{thm}[\cite{BF}, \cite{MP}]\label{t4BFMP} 
Let $G$ be a finite group with Sylow $p$-subgroup $P$.
The complete stable splitting of $BG$ is given by 
\[BG\sim \bigvee_{(Q,V)} \dim(S(G,Q,V))X_{S(P,Q,V)}\]
where $S(G,Q,V)=S(P,Q,V)[G]$.
\end{thm}

Let 
$$H^*(G)=(\bF_p \ot H^*(G,\bZ))/\sqrt{(0)}$$ 
for a finite group $G$.  
From Corollary \ref{c2deg} and Corollary \ref{c3deg}, we have the following. 

\begin{cor}\label{c4GG} 
Let $G_1$,$G_2$ have the same $p$-Sylow subgroup $E$.  
Suppose that $G_1 \leq G_2$. \\
{\rm (1)} If  
$$\dim H^{2n}(G_1)= \dim H^{2n}(G_2)$$
for all $0 \leq n \leq (p+2)(p-1)$, then $BG_1 \sim BG_2$.
\\ 
{\rm(2)} If  
$$\dim H^{2n}(G_1,\bF_p)= \dim H^{2n}(G_2,\bF_p)$$
for all $0 \leq n \leq p^2-2$, then $BG_1 \sim BG_2$.
\end{cor}

In general, the computation of these $\dim S(G,Q,V)$ is not 
so easy.
Hence we study the way to compute it from the information 
on the cohomology $H^*(G)$. (In fact, in \cite{Y07}, most direct summands in the stable splitting of $BG$ are computed from $H^*(G)$.)

Let $\mcl{F}_G$ be the fusion system on $E$ determined 
by $G$. Let $\mcl{F}_G^{ec}$-rad be the set of 
$\mcl{F}_G^{ec}$-radical maximal elementary abelian 
$p$-subgroups of $E$. If $A$ is a maximal elementary 
abelian $p$-subgroup of $E$, then $A \in \mcl{F}_G^{ec}
\mbox{-rad}$ if and only if 
$W_G(A)=N_G(A)/C_G(A)=\Out_{\mcl{F}_G}(A) \geq SL_2(\bF_p)$ 
by \cite[Lemma 4.1]{RV}. 
Let $W_G(E)=N_G(E)/EC_G(E)=\Out_{\mcl{F}_G}(E)$.  

\begin{thm}[{\cite[Theorem 4.3]{TY96}},{\cite[Theorem 3.1]{Y07}}]\label{t4HG} 
Let $G$ have the Sylow $p$-subgroup $E$, then 
   \[H^*(G)\cong H^*(E)[G]=
H^*(E)^{W_G(E)}\cap (\cap_{A \in \mcl{F}_G^{ec}\mbox{-$\rad$}} 
(\res^E_A)^{-1}(H^*(A)^{W_G(A)})).\]
Moreover, 
if $M$ is an $A_p(E,E)$-submodule of $H^*(E)$,  
then 
\[ M[G]=M^{W_G(E)}\cap
(\cap_{A \in \mcl{F}_G^{ec}\mbox{-$\rad$}} 
(\res^E_A)^{-1}(H^*(A)^{W_G(A)})). \]
\end{thm}

\begin{proof}
The first part follows from Alperin's fusion theorem 
(\cite[Theorem A.10]{BLO}). 
Let $M$ be an $A_p(E,E)$-submodule of $H^*(E)$. Since 
$[G][G] \in \Fp [G]$, it follows that 
$M[G]=M \cap H^*(E)[G]$. Hence the result follows from 
the first part. 
\end{proof}

Let $X_{i,q}$ be the indecomposable summand in the 
stable splitting of $BE$ which corresponds to the 
simple $A_p(E,E)$-module $S(E,E,S^i \ot \det^q)$.  
For $0\leq q \leq p-2$, let $L(2,q)$ (resp. $L(1,q)$) be 
the summand which corresponds to the simple 
$A_p(E,E)$-module $S(E,A,S(A)^{p-1}\ot {\det}^q)$ 
(resp. $S(E,C_p,U_q)$). 
We set $M(2)=L(1,0)\vee L(2,0)$. 

Suppose that  
the stable splitting of $BG$ is written as 
\[ BG \sim (\vee_{i,q}  n(G)_{i,q} X_{i,q})  \vee(\vee_{q} m(G,2)_q L(2,q))
\vee(\vee_{q} m(G,1)_q  L(1,q)). \]
Recall that 
\[H^{2q}(E)\cong 
\begin{cases}
S(E,C_p,U_i) & (1\le q\le p-2) \\
S(E,C_p,U_0) & (q=p-1)          
\end{cases}\]
by Theorem \ref{t2HY}. 
Hence, 
 
\begin{lem}[{\cite[Corollary 4.6]{Y07}}] \label{l4L1q} 
The multiplicity $m(G,1)_q$ for $L(1,q)$ is given by 
\[m(G,1)_q= 
\begin{cases}
\dim H^{2q}(G) & (1\le q\le p-2) \\
\dim H^{2(p-1)}(G) & (q=0).          
\end{cases} \]
\end{lem}

The multiplicity $n(G)_{i,q}$ of $X_{i,q}$ depends only on $W_G(E)=N_G(E)/EC_G(E)$. 
For $H \leq \GL_2(\Fp)$ and $\GL_2(\Fp)$-submodule $M$ of $H^*(E)$, 
let $$M^H=\{m \in M~|~\mbox{$mh=m$ for any $h \in H$}\}$$ denotes the subspace consists of $H$-invariant elements.  
Then we have the following lemma. 
  
\begin{lem}[{\cite[Lemma 4.7]{Y07}}]\label{l4nG} 
The  multiplicity  $n(G)_{i,q}$ of $X_{i,q}$ 
in  $BG$ is given by 
\[ n(G)_{i,q}=\dim (S^iv^q)^{W_G(E)}.\]
\end{lem}

Next problem is to seek the multiplicity $m(G,2)_q$ for the 
summand $L(2,q)$ in $BG$.
We can prove,

\begin{lem}[{\cite[Proposition 4.9]{Y07}}]\label{l4L20} 
The multiplicity of $L(2,0)$ in $BG$ is given by
\[m(G,2)_0= \sharp_G(A)-\sharp_G(F^{ec}A)\]
where $\sharp_G(A)$ (resp.$\sharp_G(F^{ec}A)$) is 
the number of $G$-conjugacy classes
of rank two elementary abelian $p$-subgroups in $E$ (resp. subgroups in $\mcl{F}_G^{ec}\mbox{-$\rad$}$).
\end{lem}

\begin{lem}[{\cite[Corollary 4.10]{Y07}}]\label{l4L10}
 The multiplicity of $L(1,0)$ in $BG$ 
is given by
\[m(G,1)_0=\dim  H^{2(p-1)}(G)= \sharp_G(A)-\sharp_G(F^{ec}A).\]
\end{lem}

\begin{rem}\label{r4deg}
{\rm 
By Lemma \ref{l4L20} and \ref{l4L10},  
$m(G,1)_0=m(G,2)_0$, namely,  
$L(1,0)$ and $L(2,0)$ always appear in $BG$ as 
$M(2)=L(1,0)\vee L(2,0)$. 
On the other hand, in Corollary \ref{c2deg}, all simple modules except for $S(E,A,S^{p-1})$ appear in $H^{2n}(E)$ for $n \leq p^2-1$. 
Note that the minimal $n$ such that  
$S(E,E,S^{p-1})$ appears in $H^{2n}(E)$ is 
$p^2-1 =\frac{1}{2}\deg (VS^{p-1})$. 
 Hence we may replace the bound 
$(p+2)(p-1)$ by $p^2-1$ in Corollary \ref{c4GG} (1). 
}
\end{rem}

For the number $m(G,2)_q$ for $q\not= 0$, it seems that there 
is not a good way to find it.  However we give some condition such that $m(G,2)_q=0$.

\begin{lem}[{\cite[Lemma 4.11]{Y07}}]\label{l4xi}
Let $\xi\in \bF_p^*$ be a primitive $(p-1)$-th root of $1$. 
Suppose that $G\supset E{\colon}{\langle}
\diag(\xi,\xi){\rangle}$. If $\xi^{3k}\not =1$, 
then $BG$ does not contain the summand $L(2,k)$, i.e., $m(G,2)_k=0$.
\end{lem}

Let $\xi$ be the multiplicative generator of $\Fp^*$ as 
above. Let 
$$w=\begin{pmatrix}0&1\\1&0\end{pmatrix}\in \GL_2(\Fp).$$ 
Let  
\[T=\la \mathrm{diag}(\xi,\xi),\mathrm{diag}(\xi,1)\ra \]
be a torus in $\GL_2(\Fp)$. 
Then $w$ normalizes $T$. Let 
$T\la w\ra$ be the semidirect product of $T$ by $\la w \ra$

\begin{lem}\label{l4invSvT} 
Assume that $1 \leq l \leq p-1$, $0 \leq k\leq p-2$. Then 
\[(S^lv^k)^T=
\begin{cases}
\Fp y_1^iy_2^iv^{p-1-i} & (l=2i,\ k=p-1-i,\ 1 \leq i \leq \frac{p-1}{2}) \\
\Fp  y_1^{p-1} \op \Fp y_2^{p-1} & (l=p-1,\ k=0) \\
0 & (\mbox{otherwise})
\end{cases} 
\]
and 
\[(S^lv^k)^{T\la w \ra}=
\begin{cases}
\Fp y_1^{2j}y_2^{2j}v^{p-1-2j} & (l=4j,\ k=p-1-2j,\ 1 \leq j \leq \frac{p-1}{4}) \\
\Fp  (y_1^{p-1}+y_2^{p-1}) & (l=p-1,\ k=0) \\
0 & (\mbox{otherwise}).\end{cases}\]
\end{lem}

\begin{proof}
We consider the action of 
$\diag(\xi,1)$ and $\diag(1,\xi)$. For 
$0 \leq i \leq l$, we have 
$$\diag(\xi,1)y_1^iy_2^{l-i}v^k=\xi^{i+k}y_1^iy_2^{l-i}v^k$$
$$\diag(1,\xi)y_1^iy_2^{l-i}v^k=\xi^{l-i+k}y_1^iy_2^{l-i}v^k.$$
Hence $y_1^iy_2^{l-i}v^k$ is $T$-invariant if and only if 
$i+k \equiv l-i+k \equiv 0 \mod p-1$, namely, 
$l \equiv 2i \mod p-1$ and $i+k \equiv 0 \mod p-1$. 
If $k=0$, then $i=0,p-1$, $l=p-1$. If 
$k>0$, then $i=p-1-k$, $l=2i$. Since 
$1 \leq l \leq p-1$, it follows that 
$ i \leq \frac{p-1}{2}$. 

Next consider the action of $w$. Since 
$w$ interchanges $y_1$ and $y_2$, $wv=-v$, 
$y_1^iy_2^iv^{p-1-i}$ is $w$-invariant if and only if 
$i$ is even. 
\end{proof}

Now assume that $p-1=3m$. Note that $m$ is even. 
We set 
\[H=\la\diag(\xi,\xi),\diag (\xi^3,1)\ra. \]
Then $w$ normalizes $H$. Let 
$H\la w \ra$ be the semidirect 
product of $H$ by $\la w\ra$.  

\begin{lem}\label{l4invSvH} 
Assume that $p-1=3m$. Let $m=2n$. 
Assume that $1\leq l \leq p-1$, $0 \leq k\leq  p-2$. 
Then $(S^lv^k)^H$ is equal to the following vector space. 
\begin{eqnarray*}
\begin{cases}
\Fp\{y_1^{p-1}, y_1^{2m}y_2^m, y_1^my_2^{2m}, y_2^{p-1}\} 
& (l=p-1,\ k=0) \\
\Fp\{ y_1^{i-n}y_2^{i+n}v^{3n-i}, y_1^{i+n}y_2^{i-n}v^{3n-i}\} & 
(l=2i, \ k=3n-i, \ n \leq i < 3n) \\
\Fp\{y_1^iy_2^iv^{3m-i}\} & 
(l=2i,\ k=(p-1)-i, \ 1 \leq i<m) \\ 
\Fp\{y_1^{i-m}y_2^{i+m}v^{3m-i}, 
y_1^iy_2^iv^{3m-i}, y_1^{i+m}y_2^{i-m}v^{3m-i}\} & 
(l=2i,\ k=(p-1)-i, \ m\leq i \leq 3n) \\
0 & (\mbox{otherwise}).
\end{cases} 
\end{eqnarray*}
\end{lem}

\begin{proof}
We consider the action of 
$\diag(\xi,\xi)$ and $\diag(\xi^3,1)$ 
on $y_1^j y_2^{l-j}v_k$ for $0 \leq j\leq l$. 
Since 
$$\diag(\xi,\xi)y_1^jy_2^{l-j}v^k=\xi^{l+2k}y_1^jy_2^{l-j}v^k$$
and 
$$\diag(\xi^3,1)y_1^jy_2^{l-j}v^k=\xi^{3j+3k}y_1^jy_2^{l-j}v^k , $$
$y_1^jy_2^{l-j}v^k$ is $H$-invariant if and only if 
$$\begin{cases}
l+2k \equiv 0 &\mod p-1 \\
j+k   \equiv 0 &\mod m. \end{cases}$$
Note that the condition $l+2k \equiv 0 \mod p-1$ 
implies that $l$ is even, so we set $l=2i$, 
$1 \leq i \leq 3n$. 
Then $y_1^jy_2^{l-j}v^k \in (S^lv^k)^H$ if and only if 
$$\begin{cases}
i+k \equiv 0 &\mod 3n \\
j+k   \equiv 0 &\mod m. \end{cases}$$
Since $0 \leq k \leq p-2$, 
$k \equiv -i \mod 3n$ if and only if  
$$k=3n-i \ \mbox{or}  \ (p-1)-i.$$
 
First, assume $k=3n-i$. Then 
$j \equiv -k=i-3n \mod m$ if and only if  
$j=i+s n$ for some integer $s \in \bZ$ such that 
$s \equiv 1 \mod 2$. 
Then 
\begin{eqnarray*}
0 \leq j\leq l=2i &\iff& 0\leq i+s n \leq 2i \\
&\iff& -i  \leq sn \leq i \\
&\iff& |s|\leq \frac{i}{n}. 
\end{eqnarray*}
If $1 \leq i < n$, namely, $\frac{i}{n}<1$, then 
there is no $s \in \bZ$ such that 
$|s|\leq \frac{i}{n}$ with $s \equiv 1 \mod 2$. 

Assume that $n \leq i \leq 3n$. 
If $i=3n$, namely, $l=p-1$ and $k=0$, then, since $\frac{i}{n}=3$, 
$|s|\leq \frac{i}{n}$ if and only if 
$s=-3,-1,1, 3$. Then  
$$j=i+s n=(3+s)n=0,m,2m,3m.$$
If $n \leq i < 3n$, namely, 
$1 \leq \frac{i}{n} <3$, then 
$|s|\leq \frac{i}{n}$ if and only if $s=-1,1$ since 
$s \equiv 1 \mod 2$. So we have 
$$j=i+s n=i\pm n.$$

Next we consider the case $k=(p-1)-i$. Then, 
$$j\equiv -k=i-(p-1) \mod m$$
if and only if $j=i+sm$ for some $s \in \bZ$. Then, 

\begin{eqnarray*}
0 \leq j=i+sm \leq l=2i &\iff& -i \leq sm \leq i \\
&\iff& |s| \leq \frac{i}{m}. 
\end{eqnarray*}
If $1 \leq i <m$, then this implies $s=0$ and $j=i$. 
If $m \leq i \leq 3n$, namely, 
$1 \leq \frac{i}{m} \leq \frac{3}{2}<2$, then 
this implies $s=0,\pm1$. 
Hence we have $j=i,i\pm m$. This completes the proof.
\end{proof}

\begin{lem}\label{l4invSvHw} 
Assume that $p-1=3m$. Let $m=2n$. 
Assume that $1\leq l \leq p-1$, $0 \leq k\leq  p-2$. Then 
$(S^lv^k)^{H\la w\ra}$ is equal to the following vector space. 
\begin{eqnarray*}
\begin{cases}
\Fp\{y_1^{p-1}+y_2^{p-1}, y_1^{2m}y_2^m+y_1^my_2^{2m}\} 
& (l=p-1,\ k=0) \\
\Fp\{ (y_1^{i-n}y_2^{i+n}+(-1)^{3n-i} y_1^{i+n}y_2^{i-n})v^{3n-i}\} & 
(l=2i, \ k=3n-i, \ n \leq i < 3n) \\
\Fp\{y_1^iy_2^iv^{3m-i}\} & 
(l=2i,\ k=(p-1)-i, \ 1 \leq i<m, \mbox{$i$:even})
\\ 
\Fp\{(y_1^{i-m}y_2^{i+m}+y_1^{i+m}y_2^{i-m})v^{3m-i}, 
y_1^iy_2^iv^{3m-i} \} & 
(l=2i,\ k=(p-1)-i, \ m\leq i \leq 3n, \mbox{$i$:even}) \\
\Fp\{(y_1^{i-m}y_2^{i+m}-y_1^{i+m}y_2^{i-m})v^{3m-i}\} & 
(l=2i,\ k=(p-1)-i, \ m\leq i \leq 3n, \mbox{$i$:odd}) \\
0 & (\mbox{otherwise}).
\end{cases} 
\end{eqnarray*}
\end{lem}

\begin{proof} 
Since $w$ interchanges 
$y_1$ and $y_2$, $vw=(\det w)v=-v$,  
this lemma follows from the previous lemma. 
\end{proof}

Let $p-1=3m$. We consider the multiplicity 
$m(G,2)_m$ and $m(G,2)_{2m}$ in some cases. 
Recall that $(CS^q+T^q)v^q\cong S(E,A,S(A)^{p-1}\ot {\det}^q)$ 
for $1 \leq q \leq p-2$ by \cite[Corollary 9.3]{HY}. This module has a basis
$$(y_1^iy_2^j)v^q,~ (i=0,~ q \leq i \leq p-1+q,~j=p-1+q-i).$$ 
Note that 
$$y_1^iy_2^j,~(i=0,~ p \leq i \leq p-1+q,~j=p-1+q-i)$$ 
is a basis of $CS^q$. On the other hand, 
$$y_1^iy_2^j,~  (q \leq i \leq p-1~, j=p-1+q-i)$$
is a basis of $T^q$.  Moreover, if $q=m$ or $q=2m$, then  
all elements in $(CS^q+T^q)v^q$ are 
$\mathrm{diag}(\xi,\xi)$-invariant since 
$$\mathrm{diag}(\xi,\xi)(y_1^iy_2^jv^q)=
\xi^{i+j+2q}(y_1^iy_2^jv^q)=\xi^{p-1+3q}(y_1^iy_2^jv^q)
=(y_1^iy_2^jv^q).$$

\begin{lem}\label{l4invCSm} 
Let $M=(CS^m+T^m)v^m$. \\
{\rm(1)} $M^T$ has a basis 
$$y_1^{2m}y_2^{2m}v^m.$$
{\rm(2)} $M^H$ has a basis 
$$y_1^{4m}v^m, \ y_1^{3m}y_2^{m}v^m, \ y_1^{2m}y_2^{2m}v^m, \  
y_1^my_2^{3m}v^m, \ y_2^{4m}v^m.$$
{\rm(3)} $M^{T\la w\ra}$ has a basis
$$y_1^{2m}y_2^{2m}v^m.$$
{\rm(4)} $M^{H\la w\ra}$ has a basis 
$$(y_1^{4m}+y_2^{4m})v^m=C(y_1^m+y_2^m)v^m, \ 
y_1^my_2^m(y_1^{2m}+y_2^{2m})v^m, \ 
y_1^{2m}y_2^{2m}v^m.$$
\end{lem}

\begin{proof}
Since
$$\diag(\xi,1)(y_1^iy_2^jv^m)=
\xi^{i+m}(y_1^iy_2^jv^m)$$
and
$$\diag(\xi^3,1)(y_1^iy_2^jv^m)=
\xi^{3(i+m)}(y_1^iy_2^jv^m)=\xi^{3i}(y_1^iy_2^jv^m),$$
$y_1^iy_2^jv^m$ is $T$-invariant if and only if 
$i+m \equiv 0 \mod p-1$. Moreover, 
$y_1^iy_2^jv^m$ is $H$-invariant 
$i \equiv 0 \mod m$. 
\\ 
(1) Since $i=0$ or $m \leq i \leq p-1+m=4m$, 
$i+m \equiv 0 \mod p-1$ if and only if 
$i=2m$. 
\\
(2) Since $i=0$ or $m \leq i \leq p-1+m=4m$, 
$i \equiv 0 \mod m$ if and only if 
$i=0, m,2m,3m,4m$. 
\\
(3) (4) 
Since $m$ is even, $w$ acts on $v^m$ trivially. On the 
other hand $w$ interchanges $y_1$ and $y_2$.  Hence 
the results follows from (1) and (2).  
\end{proof}

Similarly, we have the following. 

\begin{lem}\label{l4invCS2m} 
Let $M=(CS^{2m}+T^{2m})v^{2m}$. \\
{\rm(1)} $M^T$ has a basis 
$$Cy_1^{m}y_2^mv^{2m}=y_1^{4m}y_2^{m}v^{2m}.$$
{\rm(2)} $M^H$ has a basis 
$$y_1^{5m}v^{2m}, \ Cy_1^my_2^mv^{2m}=y_1^{4m}y_2^{m}v^{2m}, 
\ y_1^{3m}y_2^{2m}v^m, \  
y_1^{2m}y_2^{3m}v^m, \ y_2^{5m}v^m.$$
{\rm(3)} $M^{T\la w\ra}$ has a basis 
$$Cy_1^my_2^mv^{2m}=y_1^{4m}y_2^mv^{2m}.$$
{\rm(4)} $M^{H\la w\ra}$ has a basis 
$$(y_1^{5m}+y_2^{5m})v^{2m}=C(y_1^{2m}+y_2^{2m})v^{2m}, \ 
y_1^{2m}y_2^{2m}(y_1^{m}+y_2^{m})v^{2m}, \ 
Cy_1^my_2^mv^{2m}.$$
\end{lem}

If $A \in \mcl{A}(E)$ is a maximal elementary abelian 
$p$=subgroup of $E$, then 
$$H^*(A)=\Fp[y_A, u_A],\ \deg y_A=\deg u_A=2.$$
We may assume that
$$\res^E_{A_i}(y_1)=y_{A_i},\ 
\res^E_{A_i}(y_2)=iy_{A_i} \  \mbox{for} \ i \in \Fp$$
and
$$\res^E_{A_{\infty}}(y_1)=0, \ 
\res^E_{A_{\infty}}(y_2)=y_{A_{\infty}}.$$
Moreover, 
$$\res^E_A(C)=y_A^{p-1}, \ 
\res^E_{A}(v)=u_A^p-y_A^{p-1}u$$
for any $A \in \mcl{A}(E)$ (see \cite[section 4]{HY}). 

\begin{lem}\label{l4CSG} 
Let $1 \leq q \leq p-2$. Then 
$$((CS^q+T^q)v^q)[G]= 
((CS^q+T^q)v^q)^{W_G(E)} \cap 
(\cap_{A \in \mcl{F}^{ec}_G\mbox{-$\rad$}} \ker \res^E_A)$$
\end{lem}

\begin{proof} 
Let $y=y_A$ and $u=u_A$ for $A \in \mcl{A}(E)$. Then  
$$\res^E_A((CS^q+T^q)v^q)=\Fp y^{p-1+q}\res^E_A(v^q)=
\Fp y^{p-1}(yu^p-y^pu)^q.$$ 
If $g \in \Aut(A)=\GL_2(\Fp)$, then 
$$g(yu^p-y^pu)=(\det g)(yu^p-y^pu)$$
and $y^{p-1}(yu^p-y^pu)^q$ is not 
$\mathrm{SL}_2(\Fp)$-invariant, hence 
the result follows from Theorem \ref{t4HG}.
\end{proof}

\begin{prop}\label{p4mG} 
Suppose that $\mcl{F}^{ec}_G\mbox{-$\rad$}=\{A_0,A_{\infty}\}$. 
Then $m(G,2)_m=m(G,2)_{2m}$ and we have the following values:
$$\begin{array}{c|cccc}
W_G(E) &H   &  H\la w \ra & T & T\la w \ra \\ \hline
m(G,2)_m=m(G,2)_{2m}& 3&2&1&1
\end{array}$$
\end{prop}

\begin{proof}
Since 
$$
\res^E_{A_0}(y_1)\ne 0, \quad \res^E_{A_0}(y_2)=0$$
and 
$$\res^E_{A_{\infty}}(y_1)=0,\quad \res^E_{A_{\infty}}(y_2)\ne 0,$$
the results follows from Lemma \ref{l4invCSm},  
Lemma \ref{l4invCS2m} and \ref{l4CSG}. 
\end{proof}

Next we study the stable splitting of $BG$ for 
some $G$ related to the linear group $L_3(p)$. 
There are 6 saturated fusion systems related to $L_3(p)$ \cite[p. 46, Table 1.1]{RV}.
\vspace{.5cm}
\begin{center}
\begin{tabular}{cclc}
$W_G(E)$   &  $|\mcl{F}_G^{ec}{\mbox{-$\rad$}}|$ & Group & $p$ \\ \hline\hline
$H$              & $1+1$ & $L_3(p)$     & $3\mid(p-1)$\\ \hline
$H\la w \ra$ & $2$    & $L_3(p):2$   & $3\mid(p-1)$\\ \hline
$T$              & $1+1$ & $L_3(p).3$   & $3\mid(p-1)$\\ \hline
$T\la w \ra$ & $2$    & $L_3(p).S_3$ & $3\mid(p-1)$\\ \hline
$T$              & $1+1$ & $L_3(p)$     & $3 \nmid (p-1)$ \\ \hline
$T\la w \ra$ & $2$    & $L_3(p):2$   & $3 \nmid (p-1)$ \\ \hline
\end{tabular}
\end{center}
\vspace{.5cm}
We determine the stable splittings of these 6 groups. 
Note that the these results, with the results in \cite{Y07}, give a complete information on the splitting for 
fusion systems on $E$ with 
$|\mcl{F}^{ec}{\mbox{-$\rad$}}| \geq 2$  
by the classification in \cite{RV}. 

Let 
\[ X= X_{0,0} 
\vee 2X_{p-1,0}\vee (\vee_{1\le i \leq (p-1)/2}X_{2i,p-1-i})
\vee M(2) \]
and 
\[ X'= X_{0,0} 
\vee X_{p-1,0}\vee (\vee_{1\le j \leq (p-1)/4}X_{4j,p-1-2j})
\vee M(2).\]

\begin{thm}\label{t4WT} 
Suppose that 
$W_G(E)=T$ and 
$\mcl{F}_G^{ec}{\mbox{-$\rad$}}=\{A_0,A_{\infty}\}$. 
If $3 \nmid p-1$ then 
$BG$ is stably homotopic to $X$. 
If $p-1=3m$, then 
$BG$ is stably homotopic to $X \vee L(2,m) \vee L(2,2m)$. 
\end{thm}

\begin{proof}
By Lemma \ref{l4L1q} and Lemma \ref{l4invSvT}, $L(1,q)$ 
($1 \leq q \leq p-2$) is not contained in $BG$. 
Since $A_i$ $(1 \leq i \leq p-1)$ are $T$-conjugate, 
there are three conjugacy classes of maximal elementary 
abelian $p$-subgroups and 
two of them consist of  $\mcl{F}_G^{ec}$-radical subgroups.  
From  Lemma \ref{l4L20} and Lemma \ref{l4L10},  
just one $L(2,0)$ (and one $L(1,0)$) is contained in $BG$.
Moreover if $3$ does not divide $p-1$, then  $L(2,q)$ 
is not contained in $BG$ for each $1 \leq q \leq p-2$ from 
Lemma \ref{l4xi}.

Moreover, by Lemma \ref{l4nG} and Lemma \ref{l4invSvT}, 
$$n(G)_{l,k}=
\begin{cases}
1 & (l=2i,\ k=p-1-i,\ 1 \leq i \leq \frac{p-1}{2}) \\
2 & (l=p-1,\ k=0) \\
0 & (\mbox{otherwise}).
\end{cases}$$ 

On the other hand, if $p-1=3m$, then 
$m(G,2)_m=m(G,2)_{2m}=1$ by Proposition \ref{p4mG}. 
This completes  the proof. 
\end{proof}

\begin{thm}\label{t4WTw} 
Suppose that 
$W_G(E)=T\la w \ra$ and 
$\mcl{F}_G^{ec}{\mbox{-$\rad$}}=\{A_0,A_{\infty}\}$. 
If $3 \nmid p-1$ then 
$BG$ is stably homotopic to $X'$. 
If $p-1=3m$, then 
$BG$ is stably homotopic to $X' \vee L(2,m) \vee L(2,2m)$. 
\end{thm}

\begin{proof}
The proof is similar to that of previous Theorem. 
By Lemma \ref{l4L1q} and Lemma \ref{l4invSvT}, $L(1,q)$ 
($1 \leq q \leq p-2$) is not contained in $BG$. 
Since $A_i$ $(1 \leq i \leq p-1)$ are $T$-conjugate, 
there are two conjugacy classes of maximal elementary 
abelian $p$-subgroups and 
one of them consists of  $\mcl{F}_G^{ec}$-radical subgroups.  
From  Lemma \ref{l4L20} and Lemma \ref{l4L10},  
just one $L(2,p-1)$ (and one $L(1,p-1)$) is contained in $BG$.
Moreover if 
$3$ does not divide $p-1$, 
then  $L(2,q)$ 
is not contained in $BG$ for each $1 \leq q \leq p-2$ from 
Lemma \ref{l4xi}.

Moreover, by Lemma \ref{l4nG} and Lemma \ref{l4invSvT}, 
$$n(G)_{l,k}=
\begin{cases}
1 & (l=4j,\ k=p-1-2j,\ 1 \leq j \leq \frac{p-1}{4}) \\
1 & (l=p-1,\ k=0) \\
0 & (\mbox{otherwise}).
\end{cases}$$ 

On the other hand, if $p-1=3m$, then 
$m(G,2)_m=m(G,2)_{2m}=1$ by Proposition  \ref{p4mG}. 
This completes  the proof. 
\end{proof}

Next assume that $p-1=3m$. 
Let $m=2n$. 

\begin{thm}\label{t4WH} 
Suppose that $W_G(E)=H$ and 
$\mcl{F}_G^{ec}{\mbox{-$\rad$}}=\{A_0,A_{\infty}\}$. Then 
$BG$ is stably homotopic to 
$$
X_{0,0} \vee 4X_{p-1,0} 
\vee 2(\vee_{n \leq i <3n}X_{2i, 3n-i}) \vee
(\vee_{1 \leq i<m}X_{2i, 3m-i})$$
$$ \vee 
3(\vee_{m \leq i\leq 3n}X_{2i,3m-i})
\vee 3(M(2) \vee L(2,m) \vee L(2,2m)).
$$
\end{thm}

\begin{proof} 
By Lemma \ref{l4L1q} and Lemma \ref{l4invSvH}, 
$m(G,1)_q=0$ for $1 \leq q \leq p-2$. On the other hand, 
by Lemma \ref{l4L20} and Corollary \ref{l4L10}, 
$m(G,1)_0=m(G,2)_0=5-2=3$. 
By Lemma \ref{l4xi}, $m(G,2)_k=0$ for $1 \leq k \leq p-2$, 
$k \ne m,2m$.
By Proposition \ref{p4mG}, $m(G,2)_m=m(G,2)_{2m}=3$. 
The multiplicity $n(G)_{i,q}$ is obtained by Lemma \ref{l4invSvH}. 
By Lemma \ref{l4invSvH}, 
$$n(G)_{i,q}=
\begin{cases}
4 & (l=p-1,q=0)\\
2 & (l=2i,q=3n-i, n\leq i<3n) \\
1 & (l=2i, q=(p-1)-i, 1\leq i <m) \\
3 & (l=2i, q=(p-1)-i, m\leq i \leq 3n)\\
0 & (\mbox{otherwise}).
\end{cases}$$
\end{proof}

\begin{thm}\label{t4WHw} 
Suppose that $W_G(E)=H \la w \ra$ and 
$\mcl{F}_G^{ec}{\mbox{-$\rad$}}=\{A_0,A_{\infty}\}$. 
Then $BG$ is stably homotopic to 
$$
X_{0,0} \vee 2X_{p-1,0}  
\vee (\vee_{n \leq i <3n}X_{2i, 3n-i}) \vee
(\vee_{1 \leq j \leq 3n/2}X_{4j, 3m-2j}) \vee 
(\vee_{m\leq i \leq 3n}X_{2i,(p-1)-i})$$
$$
\vee 2(M(2) \vee L(2,m) \vee L(2,2m)).$$
\end{thm}

\begin{proof} By Lemma \ref{l4L1q} 
and Lemma \ref{l4invSvHw},
$m(G,1)_q=0$ for $1 \leq q \leq p-2$. 
By Lemma \ref{l4L20} and Corollary \ref{l4L10}, 
$m(G,1)_0=m(G,2)_0=3-1=2$. 
By Lemma \ref{l4xi}, $m(G,2)_k=0$ for $1 \leq k \leq p-2$, 
$k \ne m,2m$.
By Proposition \ref{p4mG}, $m(G,2)_m=m(G,2)_{2m}=2$. 
The multiplicity $n(G)_{i,q}$ is obtained by Lemma \ref{l4invSvHw}. 
By Lemma \ref{l4invSvHw}, 
$$n(G)_{i,q}=
\begin{cases}
2 & (l=p-1,q=0)\\
1 & (l=2i,q=3n-i, n\leq i<3n) \\
1 & (l=2i, q=(p-1)-i, 1\leq i <m, \ \mbox{$i$: even}) \\
2 & (l=2i, q=(p-1)-i, m \leq i \leq 3n, \ \mbox{$i$: even})\\
1 & (l=2i, q=(p-1)-i, m \leq i \leq 3n, \ \mbox{$i$: odd})\\
0 & (\mbox{otherwise}).
\end{cases}$$
Moreover, consider the 3rd, 4th and 5th cases. We 
have 
\begin{eqnarray*}
&&(\vee_{1 \leq i<m, i:{\rm even}}X_{2i,(p-1)-i}) \vee 
2(\vee_{m \leq i \leq 3n, i:{\rm even}}X_{2i,(p-1)-i}) \vee
(\vee_{m \leq i\leq3n,i;{\rm odd}}X_{2i, (p-1)-i}) \\
&=&
(\vee_{1 \leq i \leq 3n, i:{\rm even}}X_{2i,(p-1)-i}) \vee 
(\vee_{m\leq i \leq 3n}(X_{2i,(p-1)-i}) \\
&=&
(\vee_{1 \leq j \leq 3n/2}X_{4j,(p-1)-2j}) \vee 
(\vee_{m\leq i \leq 3n}X_{2i,(p-1)-i}). 
\end{eqnarray*}
This completes the proof.
\end{proof}

Next we consider the specific case, that is, $p=7$. 
We give a result which supplements the result on 
splitting for $p=7$ in \cite{Y07}. 
 
\begin{example}\label{e4p7} 
Let $p=7$, $p-1=6$, $m=2$, $n=1$. 
Suppose that 
$\mcl{F}_G^{ec}{\mbox{-$\rad$}}=\{A_0,A_{\infty}\}$.\\
{\rm (1)} If $W_G(E)=T$, then 
$$BG \sim X_{0,0} 
\vee X_{2,5} \vee X_{4,4} \vee 2X_{6,0} \vee X_{6,3}  
\vee M(2) \vee L(2,2) \vee L(2,4).$$
{\rm (2)} If $W_G(E)=T\la w\ra$, then 
$$BG \sim 
X_{0,0} \vee X_{4,4} \vee X_{6,0}  
\vee M(2) \vee L(2,2) \vee L(2,4).$$
{\rm (3)} If $W_G(E)=H$, then  
$$BG \sim X_{0,0} \vee 
2X_{2,2} \vee X_{2,5} \vee 2X_{4,1}  \vee 3X_{4,4} \vee 4X_{6,0} 
\vee 3X_{6,3}$$
$$\vee
3(M(2) \vee L(2,2) \vee L(2,4)).$$
{\rm (4)} If $W_G(E)=H\la w \ra$, then   
$$BG \sim X_{0,0} \vee X_{2,2} \vee X_{4,1} \vee 2X_{4,4} \vee 
2X_{6,0} \vee X_{6,3}$$
$$\vee 2(M(2) \vee L(2,2) \vee L(2,4)).$$
\end{example}

Let $G_1$ and $G_2$ be finite groups with Sylow $p$-subgroup 
$E$. If  $\mcl{F}_{G_1}$ is (isomorphic to) a subfusion system of 
$\mcl{F}_{G_2}$, then $BG_1 \sim BG_2 \vee X$ for 
some summand $X$ of $BG_1$.  
In this case, we write 
$$\begin{CD}G_2 @<{X}<< G_1\end{CD}$$
We use same notation for fusion systems. 

In \cite{Y07}, the second author considered the 
graphs related to the splitting of sporadic simple groups 
and some exotic fusion system for $p=7$ and obtained 
the following. 

\begin{thm}[{\cite[Theorem 9.4]{Y07}}]\label{t4p7-1} Let $p=7$. 
We have the following two sequences:
$$\begin{CD}
X_{0,0}\vee X_{4,4}\sim RV_3 
@<{X_{2,2}\vee X_{6,0} }<<  RV_2 
@<{M(2)\vee L(2,2) \vee L(2,4)}<<  O'N:2 
@<{X_{4,1}\vee X_{4,4} \vee X_{6,0} \vee X_{6,3}}<<  O'N
\end{CD}$$

$$\begin{CD}
X_{0,0} \vee X_{4,4} \vee X_{6,0}  \sim RV_1
@<{M(2)}<< Fi_{24} 
@<{X_{2,2} \vee X_{6,0}\vee X_{6,3}}<< Fi'_{24} 
@<{M(2)\vee L(2,2) \vee L(2,4)}<< He:2 \end{CD}$$ 
$$\begin{CD}
@<{X_{3,0}\vee X_{3,3} \vee X_{5,2} \vee X_{5,5}\vee 
L(1,3) \vee L(2,3)}<< He .
\end{CD}$$
where $RV_1, RV_2, RV_3$ are the exotic fusion systems of 
Ruiz and Viruel \cite{RV}. 
\end{thm}

Now we add more information on the splittings for $p=7$. 

\begin{thm}\label{t4p7-2} Let $p=7$. 
We have the following diagram:
$$\begin{CD}
RV_1 @<{M(2) \vee \tilde{L}}<< L_3(7).S_3 @<{Y'}<< L_3(7).3 \\
@A{Y \vee Z \vee M(2) \vee \tilde{L}}AA 
@A{Y \vee Z \vee M(2) \vee \tilde{L}}AA 
@AA{2(Y \vee Z \vee M(2) \vee \tilde{L})}A\\
O'N @<{M(2) \vee \tilde{L}}<< L_3(7):2 
@<{Y \vee Y' \vee Z \vee M(2) \vee \tilde{L}}<< L_3(7)\\
@V{Y \vee Z \vee M(2) \vee \tilde{L}}VV 
@V{Y \vee Z \vee M(2) \vee 2\tilde{L}}VV 
@VV{Y \vee Y' \vee 2Z \vee 2M(2) \vee 3\tilde{L}}V\\
RV_1 @<{M(2)}<< Fi_{24} @<{Y}<< Fi'_{24}
\end{CD}$$
where 
$$
Y=X_{2,2} \vee X_{6,0} \vee X_{6,3}, \quad 
Y'=X_{2,5} \vee X_{6,0} \vee X_{6,3}, \quad 
Z=X_{4,1}\vee X_{4,4}$$
$$\tilde{L}=L(2,2) \vee L(2,4).$$
\end{thm}

\begin{proof}
We have the following table by \cite[Lemma 4.9, Lemma 4.16]{RV}. 
$$\begin{array}{clll}
\mbox{group} & \Out_{\mcl{F}}(E) & 
\mcl{F}^{ec}{\mbox{-$\rad$}} & \Out_{\mcl{F}}(A) 
\\
\mbox{(fusion system)} &&& \\ \hline \\[-.3cm]
L_3(7) & 6 \times 2=\la 3I, u\ra &\{A_0\}\{A_{\infty}\} & 
\SL_2(7):2 \\
L_3(7) & 6 \times 2=\la 3I, w\ra &\{A_1\}\{A_6\} & 
\SL_2(7):2 \\
L_3(7):2 & 
(6 \times 2):2=\la 3I,u,w \ra & 
\{A_0,A_{\infty}\}& \SL_2(7):2 \\
L_3(7):2 & (6 \times 2):2=\la 3I,u,w\ra & 
\{A_1,A_6\}& \SL_2(7):2 \\
O'N & (6 \times 2):2=\la 3I,u,w\ra & 
\{A_0,A_{\infty}\}\{A_1,A_6\}& \SL_2(7):2 \\
L_3(7).3& 6^2=T & \{A_0\} \{A_{\infty}\} & \GL_2(7)\\
L_3(7).S_3 & 6^2:2=T\la w \ra & \{A_0,A_{\infty}\} & \GL_2(7)\\ 
Fi'_{24} & 6 \times S_3 =\la 3I, s, w \ra & 
\{A_1,A_2, A_4\}\{A_3,A_5,A_6\} & \SL_2(7):2\\
Fi_{24} & 6^2:2=T\la w \ra & \{A_1, \dots, A_6\} & \SL_2(7):2\\
RV_1& 6^2:2=T\la w \ra &\{A_0,A_{\infty}\}\{A_1, \dots, A_6\} & 
\GL_2(7), \ \SL_2(7):2
\end{array}$$
where
$$I=
\begin{pmatrix}1&0 \\ 0&1\end{pmatrix}, \ 
u=\begin{pmatrix}-1&0\\ 0&1\end{pmatrix}, \ 
w=\begin{pmatrix}0&1 \\ 1&0\end{pmatrix}, \ 
s=\begin{pmatrix}2 &0 \\ 0 &4\end{pmatrix}$$
and $T=\{\diag(\al, \beta)~|~\al, \beta \in \mF_7^{\times}\}$ is 
the subgroup of all invertible diagonal matrices. 
The set $\mcl{F}^{ec}{\mbox{-$\rad$}}$ is separated by 
conjugacy classes and $\Out_{\mcl{F}}(A)$ is described 
for each representative $A$ of conjugacy classes 
in $\mcl{F}^{ec}{\mbox{-$\rad$}}$ if they are different. 
Note that if we take the generators $a$ and $b$ of $E$ 
suitably, we can obtained the two rows in the case of 
$L_3(7)$ and $L_3(7):2$. 
For example, consider $G=L_3(7)$. 
Let $E$ be the group of all upper triangular matrices with 
diagonal entry $1$. 
The subgroups in $\mcl{F}_G^{ec}{\mbox{-$\rad$}}$ 
are 
$$\left\{ \left.
\begin{bmatrix}
1&\al &\beta \\ 0&1&0 \\0&0&1
\end{bmatrix}\right| \al, \beta \in \Fp \right\}, \quad 
\left\{ \left.
\begin{bmatrix}
1&0 &\al \\ 0&1&\beta \\0&0&1
\end{bmatrix}\right| \al, \beta \in \Fp \right\}.
$$ 
If we take 
$$a=
\begin{bmatrix}1&1&0\\0&1&0\\0&0&1\end{bmatrix}, \quad 
b=\begin{bmatrix}1&0&0\\0&1&1\\0&0&1\end{bmatrix}$$
then $\mcl{F}_G^{ec}{\mbox{-$\rad$}}=\{A_0, A_{\infty}\}$ 
and $\Out_{\mcl{F}_G}(E)=\la 3I,u\ra =H$. 
On the other hand, if we take 
$$a=
\begin{bmatrix}1&1&0\\0&1&1\\0&0&1\end{bmatrix}, \quad 
b=\begin{bmatrix}1&1&0\\0&1&-1\\0&0&1\end{bmatrix}$$
then $\mcl{F}_G^{ec}{\mbox{-$\rad$}}=\{A_1, A_6\}$ 
and $\Out_{\mcl{F}_G}(E)=\la 3I,w\ra$. 

The inclusions of fusion systems are obtained by the table above.  For $Fi'_{24} \lolarr L_3(7)$ and 
$Fi_{24} \lolarr L_3(7):2$, we use the second rows 
of $L_3(7)$ and $L_3(7):2$. 
The information on the summands are obtained by 
Example \ref{e4p7} and Theorem  \ref{t4p7-1}. 
\end{proof}

\begin{rem}\label{r4ONFi}
{\rm 
As we can see from Theorem \ref{t4p7-1} or \ref{t4p7-2} above, 
$B(Fi_{24})$ is a stable summand of $B(O'N)$, 
$B(O'N) \sim B(Fi_{24}) \vee Y \vee Z \vee \tilde{L}$, but 
the fusion system of $O'N$ is not isomorphic to a subfusion 
system of fusion system of $Fi_{24}$, namely, 
$$\begin{CD}
Fi_{24} @<{Y\vee Z \vee \tilde{L}}<< O'N
\end{CD}
$$
does not hold. 

Let $\mcl{F}_0=\mcl{F}_{O'N}$, 
$\mcl{F}_1=\mcl{F}_{Fi_{24}}$. By \cite[Lemma 4.3]{RV}, 
for each $A_i \in \mcl{F}_0^{ec}{\mbox{-$\rad$}}$, 
there exists an element of order $6$ in 
$\Out_{\mcl{F}_0}(E) \leq \GL_2(\Fp)$ 
which has an eigenvalue $3$ with eigenvector 
$\begin{pmatrix}1 \\ i\end{pmatrix}$ 
($\begin{pmatrix}0 \\ 1\end{pmatrix}$ if $i=\infty$) and 
determinant $5$. Hence there exists an involution in 
$\Out_{\mcl{F}_0}(E)$ which has an eigenvalue $-1$ with eigenvector 
$\begin{pmatrix}1 \\ i\end{pmatrix}$ 
($\begin{pmatrix}0 \\ 1\end{pmatrix}$ if $i=\infty$) and 
determinant $-1$.

We may assume that $\Out_{\mcl{F}_1}(E)=6^2:2=T\la w\ra$ 
and $\mcl{F}_1^{ec}{\mbox{-$\rad$}}=\{A_1,\dots, A_6\}$ 
as above. 
Suppose that $K=\Out_{\mcl{F}_0}(E) (\cong (6\times 2):2) 
\leq \Out_{\mcl{F}_1}(E)$. Then $K$ contains exactly 
$4$ involutions with determinant $-1$. Moreover 
$K\supset \la \diag(-1,1), \diag(1,-1) \ra$ since 
$\la \diag(-1,1), \diag(1,-1) \ra \triangleleft (6^2:2)$. 
Note that 
$|\mcl{F}_0^{ec}{\mbox{-$\rad$}}|=4$. Since 
$\diag(-1,1)$ (resp. $\diag(1,-1)$) has an eigenvalue $-1$ with 
eigenvector  $\begin{pmatrix}1 \\ 0\end{pmatrix}$ 
(resp. $\begin{pmatrix}0 \\ 1\end{pmatrix}$) and 
determinant $-1$, it follows that 
$A_0,A_{\infty} \in \mcl{F}_0^{ec}{\mbox{-$\rad$}}$ for 
any choice of $K \leq 6^2:2=T\la w\ra$. Hence $\mcl{F}_0$ is not isomorphic to a 
subfusion system of $\mcl{F}_1$. 
}
\end{rem}


\section{Some remarks on the case $p=3$}
Recall that $H^3(E, \bZ)=\bF_p\{a_1, a_2\}$. 
The short exact sequence 
\[0 \lorarr \bZ \stackrel{p}{\lorarr} \bZ 
\stackrel{j}{\lorarr} \Fp \lorarr 0\] 
induces the following short exact sequence 
\[0 \lorarr H^2(E,\bZ) \stackrel{j_{\ast}}{\lorarr} 
H^2(E,\bF_p) \stackrel{\hat{\beta}}{\lorarr} H^3(E, \bZ) \lorarr 0.
\tag{5.1} \]
Hence 
there exist elements $a_1', a_2' \in H^2(E,\bF_p)$ such that 
$\hat{\beta}(a_i')=a_i$ and 
$$H^2(E,\bF_p)=\bF_p\{y_1,y_2, a_1',a_2'\}.$$
We consider the action of $\Out(E)=\GL_2(\bF_p)$. The sequence (5.1) is a sequence of 
$\bF_p\GL_2(\bF_p)$-modules and the map $Q_1\hat{\beta}$ induces an isomorphism of $\bF_p\GL_2(\bF_p)$-modules, 
$$H^2(E, \bF_p)/H^2(E,\bZ) \lorarr  H^3(E,\bZ) 
\lorarr \bF_p\{y_1v, y_2v\} \cong  S^1\ot \det.$$

Now, consider the sequence
\[\begin{CD}
& & H^{2p}(E,\bZ) @>{j_{\ast}}>> H^{2p}(E,\Fp)
@>{\hat{\beta}}>> H^{2p+1}(E,\bZ) \\
@>{p}>>H^{2p+1}(E,\bZ) @>{j_{\ast}}>> H^{2p+1}(E,\Fp).
\end{CD}\]
By taking the $p$-th power, we have an  
$\bF_p\GL_2(\bF_p)$-morphism, 
$$H^2(E, \bF_p) \lorarr H^{2p}(E,\bF_p).$$ 
Since $\beta=j_{\ast}\hat{\beta}$ is the Bockstein homomorphism, 
we have $\beta((a_i')^p)=0$. Moreover, since 
$j_{\ast}: H^{2p+1}(E,\bZ) \lorarr H^{2p+1}(E,\Fp)$ is 
injective, $\hat{\beta}((a_i')^p)=0$. Hence, it follows that 
$$(a_i')^p \in j_{\ast}(H^{2p}(E, \bZ)) \cong H^{2p}(E).$$ 
On the other hand
, since 
$$H^{2p}(E)=CS^1 +T^1+\bF_p\{v\} \lrarr 
S^1\op (S^{p-2}\ot \det) \op \det,$$ 
we see that if $p>3$ then $(a_i')^p=0$. 
Since $H^{even}(E, \bF_p)$ is generated by 
$1,a_1',a_2'$ as a module over $\bF_p\ot H^{even}(E, \bZ)$, 
it follows that 
$$H^{even}(E, \bF_p)/\sqrt{(0)}=(\bF_p\ot H^{even}(E, \bZ))/
\sqrt{(0)}$$
and in particular, 
$$H^*(E, \bF_p)/\sqrt{(0)}=(\bF_p\ot H^*(E, \bZ))/\sqrt{(0)}
=H^*(E)$$
for $p>3$. 

On the other hand, if $p=3$, then the sequence (5.1) does not 
split and $\bF_3\{y_1,y_2\}$ is the unique nontrivial 
$\bF_3\GL_2(\bF_3)$-submodule of $H^2(E,\bF_3)$
(cf. \cite[p.74]{L92}).  
This implies that $(a_i')^{p^n}\ne 0$ for any $n>0$ and 
in particular, we see that $a_i'$ is not nilpotent. 

In fact, the structure of $H^*(E,\bF_3)/\sqrt{(0)}$ is known by the result of Leary \cite[Theorem 7]{L92} and we have the following:

\begin{prop}\label{p5he3} 
Assume that $p=3$.  
Then 
$H^*(E,\bF_3)/\sqrt{0}$ is generated by 
$$y_1,~ y_2,~ a_1',~ a_2',~v$$
with
$$\deg y_i=\deg a_i'=2,~\deg v=6$$
subject to the following relations:
$$y_1^3y_2-y_1y_2^3=0$$
$$a_1'a_2'=a_1'y_1=a_2'y_2=y_1y_2,~
(a_1')^2=(a_2')^2=a_1'y_2=a_2'y_1.$$ 
\end{prop}

Since 
$H^4(E,\bF_3)/(H^4(E,\bF_3)\cap \sqrt{0})$ is spanned by 
$y_1^2,y_1y_2,y_2^2,(a_1')^2$ and $\dim_{\bF_3}H^4(E)=4$, 
$$H^4(E,\bF_3)/(H^4(E,\bF_3)\cap \sqrt{0})=H^4(E).$$ 
In particular, $(a_i')^2 \in H^4(E)$ and hence we have
$$H^*(E,\bF_3)/\sqrt{0}=H^*(E) \op \bF_3[v]
\{\bF_3 a_1'+\bF_3 a_2'\}.$$
Since $H^2(E,\bF_3)/H^2(E)\cong S^1 \ot \det$ as 
$\bF\GL_2(\bF_3)$-modules, 
$$(H^*(E,\bF_3)/\sqrt{0})/H^*(E)\cong 
\bF_3[v]\ot (S^1\ot \det)$$
as $\bF_3\GL_2(\bF_3)$-modules. 
If $Q$ is a proper subgroup of $E$, then 
$H^*(Q,\bF_3)/\sqrt{0}=H^*(Q)$. Hence 
\begin{eqnarray*}
(H^*(E,\bF_3)/\sqrt{0})A_3(Q,E)A_3(E,Q) &\subset& 
(H^*(Q,\bF_3)/\sqrt{0})A_3(E,Q) \\
&=&H^*(Q)A_3(E,Q) \subset H^*(E).
\end{eqnarray*} 
In particular,  
$(H^*(E,\bF_3)/\sqrt{0})/H^*(E)$ is annihilated by 
$A_3(Q,E)A_3(E,Q)$ for any $Q<E$.
Hence, every composition factors of 
$(H^*(E,\bF_3)/\sqrt{0})/H^*(E)$
as an $A_3(E,E)$-module is isomorphic 
to $S(E,E,S^i\ot \det^q)$ for some $i,q$ and 
we have the following: 

\begin{prop}\label{p5hh} 
$$(H^n(E,\bF_3)/\sqrt{0})/H^n(E) \cong
\begin{cases}
S(E,E,S^1\ot \det) & (n \equiv 2 \bmod 12) \\
S(E,E,S^1 ) & (n \equiv 8 \bmod 12) \\
0& (\mbox{otherwise}).
\end{cases}$$
\end{prop}

\begin{cor}\label{c5pd} 
Let $X_{0,0}$ be the summand which corresponding to the 
simple module $S(E,E,\Fp)$ and 
$e$ be the corresponding idempotent in $A_p(E,E)$. 
Then $$(H^*(E,\bF_3)/\sqrt{0})e \cong H^*(E)e \cong \bDA^+ .$$
\end{cor}

At last of this paper, we see more closely the cohomology 
$H^*(X)$ of a summand $X$ in the stable splitting of 
$BG$ with $E \in \Syl_3(G)$ in the case $p=3$. 
The lowest degree and some of the second lowest degree 
$*>0$ with $H^{2*}(X)\not =0$ 
are given as follows:
$$
\begin{array}{cclcccl}
L(1,1)&:& |S^1|=1&\quad & L(1,0)&:& |y^2| =2, \\
L(2,1)&:& |CS^1v|=6&\quad & L(2,0)&:&|S^2D_2|=10,\\
X_{0,0}&:&|V|=6& \quad &
X_{0,1}&:&|v|=3\ (|Cv|=5) \\
X_{1,0}&:&|S^1V|=7,& \quad& 
X_{1,1}&:&|T^1|=3\ (|S^1v|=4)\\ 
X_{2,0}&:&|S^2V|=8 &\quad &
X_{2,1}&:&|S^2v|=5
\end{array}$$
where $|x|=\frac{1}{2}\deg x$ for an element or a subspace of 
$H^*(E)$. 
First note that $BG$ always contains $X_{0,0}$. The lowest degree
of nonzero elements in  $H^{2*}(L(i,j))$ or $H^{2*}(X_{i,q})$, $(i,q) \ne (0,0)$ 
are all different except for $X_{0,1}$ and $X_{1,1}$.
On the other hand we see $H^4(X_{0,1})=0$ but $H^4(X_{1,1})\cong \bF_3$. Moreover $L(1,0)$ and $L(2,0)$ have same multiplicity by 
Lemma \ref{l4L20} and \ref{l4L10}.  
Hence we can count the numbers of 
$$L(1,1), M(2)=L(1,0) \vee L(2,0),X_{1,1},X_{0,1},X_{2,1},L(2,1),
X_{1,0}, X_{2,0}$$ 
from $H^{2*}(G)$ for $*=1,2,4,3,5,6,7,8$.
Thus we have the following result which is similar to 
Corollary \ref{c4GG} (1) (See Remark \ref{r4deg}).

\begin{thm}\label{t5dim} 
Let $G_1$ and $G_2$ be finite groups with same Sylow $3$-subgroup $E$.  
If 
\[\dim H^{2n}(G_1)= \dim H^{2n}(G_2)\] 
for $n \le 8$, 
then $BG_1 \sim BG_2$. 
\end{thm}

For example, let $G_1={^2F_4(2)}'$ and $G_2=J_4$.  
Then by \cite[Theorem 6.2]{Y07}, 
$$B({^2F_4(2)}') \sim BJ_4 \vee X_{2,0}.$$
Hence $H^{2n}({^2F_4(2)}')\cong H^{2n}(J_4)$ 
for $n<8$ and $\dim H^{16}({^2F_4(2)}')> \dim H^{16}(J_4)$. 
See 
\cite[section 6]{Y07} for details. 

\begin{rem}
{\rm 
If $G$ has a Sylow $3$-subgroup $E$, 
then $BG$ is homotopic to the classifying space of one of 
the groups listed in \cite[Theorem 6.2]{Y07}. Moreover 
the cohomology of each dominant summand $X_{i,j}$ of $BE$, expect for 
$X_{1,0}$ and $X_{1,1}$, is deduced from 
the cohomology of those finite groups. 

On the other hand, as we can see from the graph in 
\cite[Theorem 6.2]{Y07}, 
$X_{1,0}$ and $X_{1,1}$ always appear as 
$X_{1,0} \vee X_{1,1}$.  
We shall give an brief explanation of this fact. 
Let 
$$H=W_G(E)=N_G(E)/EC_G(E) \leq \Out(E)=\GL_2(\bF_3).$$
Note that $H$ is a $3'$-group, in fact, $2$-group. 
The multiplicity of $X_{1,j}$ in the stable splitting of $BG$ is 
equal to $\dim(S^1\ot{\det}^j)^H$. 
We have to show that 
$$\dim(S^1)^H=\dim(S^1\ot{\det})^H.$$
We may assume that $H\ne 1$. Then 
$\dim (S^1)^H=1$ if and only if $H$ is conjugate 
to the subgroup $\la \diag(1,-1) \ra$ in $\GL_2(\bF_3)$.  
Similarly $\dim(S^1 \ot \det)^H=1$ if and only if 
$H$ is conjugate 
to the subgroup $\la \diag(1,-1) \ra$ in $\GL_2(\bF_3)$. 
Hence we have
$$\dim (S^1)^H=\dim(S^1\ot{\det})^H$$
and this implies that $X_{1,0}$ and $X_{1,1}$ appear in $BG$ with 
same multiplicity.  
}
\end{rem}


\end{document}